\documentclass[12pt,a4paper]{article}%
\usepackage{amssymb}
\usepackage{amsmath}
\usepackage{amsfonts}
\usepackage[dvips]{graphicx}
\usepackage[pdftex]{hyperref}
\usepackage{pdfpages}
\usepackage{mathrsfs}%
\providecommand{\U}[1]{\protect\rule{.1in}{.1in}}
\hypersetup{
	colorlinks=true,
	citecolor = blue
}
\newtheorem{theorem}{Theorem}

\newtheorem{corollary}[theorem]{Corollary}

\newtheorem{definition}[theorem]{Definition}
\newtheorem{example}[theorem]{Example}

\newtheorem{lemma}[theorem]{Lemma}

\newtheorem{proposition}[theorem]{Proposition}
\newtheorem{remark}[theorem]{Remark}

\newenvironment{proof}[1][Proof]{\noindent\textbf{#1.} }{\ \rule{0.5em}{0.5em}}
\begin{document}

\author{Vladislav V. Kravchenko$^1$, V\'{\i}ctor A. Vicente-Ben\'{\i}{}tez$^2$\\{\small $^1$Department of Mathematics, Cinvestav, Campus Quer\'{e}taro}\\{\small Libramiento Norponiente \# 2000, }\\{\small Fracc. Real de Juriquilla, Quer\'{e}taro, Qro. C.P. 76230 M\'{e}xico }\\
	{\small $^2$Instituto de Matemáticas de la U.N.A.M. Campus Juriquilla}\\{\small Boulevard Juriquilla 3001, Juriquilla, Querétaro C.P. 076230 M\'{e}xico}
	\\{\small vkravchenko@math.cinvestav.edu.mx, va.vicentebenitez@im.unam.mx, } }
\title{Schr\"odinger equation with finitely many $\delta$-interactions: closed form, integral and series representations for solutions}
\date{}
\maketitle

\begin{abstract}
A closed form solution for the one-dimensional Schr\"{o}dinger equation with a
finite number of $\delta$-interactions
\[
\mathbf{L}_{q,\mathfrak{I}_{N}}y:=-y^{\prime\prime}+\left(  q(x)+\sum
_{k=1}^{N}\alpha_{k}\delta(x-x_{k})\right)  y=\lambda y,\quad0<x<b,\;\lambda
\in\mathbb{C}%
\]
is presented in terms of the solution of the unperturbed equation
\[
\mathbf{L}_{q}y:=-y^{\prime\prime}+q(x)y=\lambda y,\quad0<x<b,\;\lambda
\in\mathbb{C}%
\]
and a corresponding transmutation operator $\mathbf{T}_{\mathfrak{I}_{N}}^{f}$
is obtained in the form of a Volterra integral operator. With the aid of the
spectral parameter power series method, a practical construction of the image
of the transmutation operator on a dense set is presented, and it is proved
that the operator $\mathbf{T}_{\mathfrak{I}_{N}}^{f}$ transmutes the second
derivative into the Schr\"{o}dinger operator $\mathbf{L}_{q,\mathfrak{I}_{N}}$
on a Sobolev space $H^{2}$. A Fourier-Legendre series representation for the
integral transmutation kernel is developed, from which a new representation
for the solutions and their derivatives, in the form of a Neumann series of
Bessel functions, is derived.

\end{abstract}

\textbf{Keywords: } One-dimensional Schr\"odinger equation, point
interactions, transmutation operator, Fourier-Legendre series, Neumann series
of Bessel functions. \newline

\textbf{MSC Classification:} 34A25; 34A45; 46F10; 47G10; 81Q05.

\section{Introduction}

We consider the one-dimensional Schr\"{o}dinger equation with a finite number
of $\delta$-interactions
\begin{equation}
-y^{\prime\prime}+\left(  q(x)+\sum_{k=1}^{N}\alpha_{k}\delta(x-x_{k})\right)
y=\lambda y,\quad0<x<b,\;\lambda\in\mathbb{C},\label{Schrwithdelta}%
\end{equation}
where $q\in L_{2}(0,b)$ is a complex valued function, $\delta(x)$ is the Dirac
delta distribution, $0<x_{1}<x_{2}<\dots<x_{N}<b$ and $\alpha_{1}%
,\dots,\alpha_{N}\in\mathbb{C}\setminus\{0\}$. Schr\"{o}dinger equations with
distributional coefficients supported on a set of measure zero naturally
appear in various problems of mathematical physics
\cite{alveberio1,alveberio2,atkinson,barrera1,coutinio,uncu} and have been
studied in a considerable number of publications and from different
perspectives. In general terms, Eq. (\ref{Schrwithdelta}) can be interpreted
as a regular equation, i.e., with the regular potential $q\in L_{2}(0,b)$,
whose solutions are continuous and such that their first derivatives satisfy
the jump condition $y^{\prime}(x_{k}+)-y^{\prime}(x_{k}-)=\alpha_{k}y(x_{k})$
at special points \cite{kochubei, kostenko}. Another approach consists in
considering the interval $[0,b]$ as a quantum graph whose edges are the
segments $[x_{k},x_{k+1}]$, $k=0,\dots,N$, (setting $x_{0}=0$, $x_{N+1}=b$),
and the Schr\"{o}dinger operator with the regular potential $q$ as an
unbounded operator on the direct sum $\bigoplus_{k=0}^{N}H^{2}(x_{k},x_{k+1}%
)$, with the domain given by the families $(y_{k})_{k=0}^{N}$ that satisfy the
condition of continuity $y_{k}(x_{k}-)=y_{k+1}(x_{k}+)$ and the jump condition
for the derivative $y_{k+1}^{\prime}(x_{k}+)-y_{k}^{\prime}(x_{k}-)=\alpha
_{k}y_{k}(x_{k})$ for $k=1,\dots N$ (see, e.g.,
\cite{gesteszy1,kurasov1,kurasov2}). This condition for the derivative is
known in the bibliography of quantum graphs as the $\delta$-type condition
\cite{berkolaikokuchment}. Yet another approach implies a regularization of
the Schrodinger operator with point interactions, that is, finding a subdomain
of the Hilbert space $L_{2}(0,b)$, where the operator defines a function in
$L_{2}(0,b)$. For this, note that the potential $q(x)+\sum_{k=1}^{N}\alpha
_{k}\delta(x-x_{k})$ defines a functional that belongs to the Sobolev space
$H^{-1}(0,b)$. In \cite{bondarenko1,gulyev,hryniv,schkalikov} these forms of
regularization have been studied, rewriting the operator by means of a
factorization that involves a primitive $\sigma$ of the potential.

Theory of transmutation operators, also called transformation operators, is a
widely used tool in studying differential equations and spectral problems
(see, e.g., \cite{BegehrGilbert, directandinverse, levitan,
marchenko,SitnikShishkina Elsevier}), and it is especially well developed for
Schr\"{o}dinger equations with regular potentials. It is known that under
certain general conditions on the potential $q$ the transmutation operator
transmuting the second derivative into the Schr\"{o}dinger operator can be
realized in the form of a Volterra integral operator of the second kind, whose
kernel can be obtained by solving a Goursat problem for the Klein-Gordon
equation with a variable coefficient \cite{hugo2,levitan, marchenko}.
Furthermore, functional series representations of the transmutation kernel
have been constructed and used for solving direct and inverse Sturm-Liouville
problems \cite{directandinverse,neumann}. For Schr\"{o}dinger equations with
$\delta$-point interactions, there exist results about equations with a single
point interaction and discontinuous conditions $y(x_{1}+)=ay(x_{1}-)$,
$y^{\prime}(x_{1}+)=\frac{1}{a}y^{\prime}(x_{1}-)+dy(x_{1}-)$, $a,b>0$ (see
\cite{hald,yurkoart1}), and for equations in which the spectral parameter is
also present in the jump condition (see \cite{akcay,mammadova,manafuv}).
Transmutation operators have also been studied for equations with
distributional coefficients belonging to the $H^{-1}$-Sobolev space in
\cite{bondarenko1,hryniv,schkalikov}. In \cite{hugo2}, the possibility of
extending the action of the transmutation operator for an $L_{1}$-potential to
the space of generalized functions $\mathscr{D}^{\prime}$, was studied.

The aim of this work is to present a construction of a transmutation operator
for the Schr\"{o}dinger equation with a finite number of point interactions.
The transmutation operator appears in the form of a Volterra integral
operator, and with its aid we derive analytical series representations for
solutions of (\ref{Schrwithdelta}). For this purpose, we obtain a closed form
of the general solution of (\ref{Schrwithdelta}). From it, the construction of
the transmutation operator is deduced, where the transmutation kernel is
ensembled from the convolutions of the kernels of certain solutions of the
regular equation (with the potential $q$), in a finite number of steps. Next,
the spectral parameter power series (SPPS) method is developed for Eq.
(\ref{Schrwithdelta}). The SPPS method was developed for continuous
(\cite{kravchenkospps1,sppsoriginal}) and $L_{1}$-potentials (\cite{blancarte}%
), and it has been used in a piecewise manner for solving spectral problems
for equations with a finite number of point interactions in
\cite{barrera1,barrera2,rabinovich1}. Following \cite{hugo}, we use the SPPS
method to obtain an explicit construction of the image of the transmutation
operator acting on polynomials. Similarly to the case of a regular potential
\cite{neumann}, we obtain a representation of the transmutation kernel as a
Fourier series in terms of Legendre polynomials and as a corollary, a
representation for the solutions of equation (\ref{Schrwithdelta}) in terms of
a Neumann series of Bessel functions. Similar representations are obtained for
the derivatives of the solutions. It is worth mentioning that the methods
based on Fourier-Legendre representations and Neumann series of Bessel
functions have shown to be an effective tool in solving direct and inverse
spectral problems for equations with regular potentials, see, e.g.,
\cite{directandinverse,neumann, KravTorbadirect}.

In Section 2, basic properties of the solutions of (\ref{Schrwithdelta}) are
compiled, studying the equation as a distributionional sense in
$\mathscr{D}^{\prime}(0,b)$ and deducing properties of its regular solutions.
Section 3 presents the construction of the closed form solution of
(\ref{Schrwithdelta}). In Section 4, the construction of the transmutation
operator and the main properties of the transmutation kernel are developed. In
Section 5, the SPPS method is presented, with the mapping and transmutation
properties of the transmutation operator. Section 6 presents the
Fourier-Legendre series representations for the transmutation kernels and the
Neumann series of Bessel functions representations for solutions of
(\ref{Schrwithdelta}), and a recursive integral relation for the Fourier-Legendre coefficients is obtained. Finally, in Section 7, integral and Neumann series of
Bessel functions representations for the derivatives of the solutions are presented.

\section{Problem setting and properties of the solutions}

We use the standard notation $W^{k,p}(0,b)$ ($b>0$) for the Sobolev space of
functions in $L_{p}(0,b)$ that have their first $k$ weak derivatives in
$L_{p}(0,b)$, $1\leqslant p\leqslant\infty$ and $k\in\mathbb{N}$. When $p=2$,
we denote $W^{k,2}(0,b)=H^{k}(0,b)$. We have that $W^{1,1}(0,b)=AC[0,b]$, and
$W^{1,\infty}(0,b)$ is precisely the class of Lipschitz continuous functions
in $[0,b]$ (see \cite[Ch. 8]{brezis}). The class of smooth functions with
compact support in $(0,b)$ is denoted by $C_{0}^{\infty}(0,b)$, then we define
$W_{0}^{1,p}(0,b)=\overline{C_{0}^{\infty}(0,b)}^{W^{1,p}}$ and $H_{0}%
^{1}(0,b)=W_{0}^{1,2}(0,b)$. Denote the dual space of $H^{1}_{0}(0,b)$ by
$H^{-1}(0,b)$. By $L_{2,loc}(0,b)$ we denote the class of measurable functions
$f:(0,b)\rightarrow\mathbb{C}$ such that $\int_{\alpha}^{\beta}|f(x)|^{2}%
dx<\infty$ for all subintervals $[\alpha, \beta]\subset(0,b)$.

The characteristic function of an interval $[A,B]\subset\mathbb{R}$ is denoted
by $\chi_{[A,B]}(t)$. In order to simplify the notation, for the case of a
symmetric interval $[-A,A]$, we simply write $\chi_{A}$. The Heaviside
function is given by $H(t)=\chi_{(0,\infty)}(t)$. The lateral limits of the
function $f$ at the point $\xi$ are denoted by $f(\xi\pm)=\lim_{x\rightarrow
\xi{\pm}}f(x)$. We use the notation $\mathbb{N}_{0}=\mathbb{N}\cup\{0\}$. The
space of distributions (generalized functions) over $C_{0}^{\infty}(0,b)$ is
denoted by $\mathscr{D}^{\prime}(0,b)$, and the value of a distribution
$f\in\mathscr{D}^{\prime}(0,b)$ at $\phi\in C_{0}^{\infty}(0,b)$ is denoted by
$(f,\phi)_{C_{0}^{\infty}(0,b)}$.\newline

Let $N\in\mathbb{N}$ and consider a partition $0<x_{1}<\dots<x_{N}<b$ and the
numbers $\alpha_{1}, \dots, \alpha_{N}\in\mathbb{C}\setminus\{0\}$. The set
$\mathfrak{I}_{N}=\{(x_{j},\alpha_{j})\}_{j=1}^{N}$ contains the information
about the point interactions of Eq. (\ref{Schrwithdelta}). Denote
\[
q_{\delta,\mathfrak{I}_{N}}(x):= \sum_{k=1}^{N}\alpha_{k}\delta(x-x_{k}%
),\quad\mathbf{L}_{q}:= -\frac{d^{2}}{dx^{2}}+q(x), \quad\mathbf{L}%
_{q,\mathfrak{I}_{N}}:= \mathbf{L}_{q}+q_{\delta,\mathfrak{I}_{N}}(x).
\]
For $u\in L_{2,loc}(0,b)$, $\mathbf{L}_{q,\mathfrak{I}_{N}}u$ defines a
distribution in $\mathscr{D}^{\prime}(0,b)$ as follows
\[
(\mathbf{L}_{q,\mathfrak{I}_{N}}u,\phi)_{C_{0}^{\infty}(0,b)}:= \int_{0}^{b}
u(x)\mathbf{L}_{q}\phi(x)dx+ \sum_{k=1}^{N}\alpha_{k} u(x_{k})\phi(x_{k})
\quad\mbox{for } \, \phi\in C_{0}^{\infty}(0,b).
\]
Note that the function $u$ must be well defined at the points $x_{k}$, $k=1,
\dots, N$. Actually, for a function $u\in H^{1}(0,b)$, the distribution
$\mathbf{L}_{q,\mathfrak{I}_{N}}u$ can be extended to a functional in
$H^{-1}(0,b)$ as follows
\[
(\mathbf{L}_{q,\mathfrak{I}_{N}}u,v)_{H_{0}^{1}(0,b)}:= \int_{0}^{b}
\{u^{\prime}(x)v^{\prime}(x)+q(x)u(x)v(x) \}dx+ \sum_{k=1}^{N}\alpha_{k}
u(x_{k})v(x_{k})\quad\mbox{for }\, v\in H_{0}^{1}(0,b).
\]

We say that a distribution $F\in\mathscr{D}^{\prime}(0,b)$ is $L_{2}$-regular,
if there exists a function $g\in L_{2}(0,b)$ such that $(F,\phi)_{C_{0}%
^{\infty}(0,b)}=(g,\phi)_{C_{0}^{\infty}(0,b)}:=\int_{0}^{b}g(x)\phi(x)dx$ for
all $\phi\in{C_{0}^{\infty}(0,b)}$.

Denote $x_{0}=0$, $x_{N+1}=b$. We recall the following characterization of
functions $u\in L_{2,loc}(0,b)$ for which $\mathbf{L}_{q,\mathfrak{I}_{N}}u$
is $L_{2}$-regular.

\begin{proposition}
\label{propregular}  If $u\in L_{2,loc}(0,b)$, then the distribution
$\mathbf{L}_{q,\mathfrak{I}_{N}}u$ is $L_{2}$-regular iff the following
conditions hold. 

\begin{enumerate}

\item For each $k=0, \dots, N$, $u|_{(x_{k},x_{k+1})}\in H^{2}(x_{k},x_{k+1}%
)$. 

\item $u\in AC[0,b]$. 

\item The discontinuities of the derivative $u^{\prime}$ are located at the
points $x_{k}$, $k=1, \dots, N$, and the jumps are given by
\begin{equation}
\label{jumpderivative}u^{\prime}(x_{k}+)-u^{\prime}(x_{k}-)=\alpha_{j}
u(x_{k}) \quad\mbox{for    } k=1, \cdots, N.
\end{equation}

\end{enumerate}

In such case,
\begin{equation}
\label{regulardist}(\mathbf{L}_{q,\mathfrak{I}_{N}}u,\phi)_{C_{0}^{\infty
}(0,b)}=(\mathbf{L}_{q}u,\phi)_{C_{0}^{\infty}(0,b)} \quad\mbox{for all   }
\phi\in{C_{0}^{\infty}(0,b)}.
\end{equation}

\end{proposition}

\begin{proof}
Suppose that $\mathbf{L}_{q,\mathfrak{I}_{N}}u$ is $L_{2}$-regular. Then there
exists $g\in L_{2}(0,b)$ such that
\begin{equation}
\label{aux0prop1}(\mathbf{L}_{q,\mathfrak{I}_{N}}u,\phi)_{C_{0}^{\infty}%
(0,b)}=(g,\phi)_{C_{0}^{\infty}(0,b)} \quad\mbox{for all  } \phi\in
{C_{0}^{\infty}(0,b)}.
\end{equation}

\begin{enumerate}

\item Fix $k\in\{1, \dots, N-1\}$. Take a test function $\phi\in C_{0}%
^{\infty}(0,b)$ with $\operatorname{Supp}(\phi) \subset(x_{k},x_{k+1})$.
Hence
\begin{equation}
\label{auxprop1}\int_{x_{k}}^{x_{k+1}}g(x)\phi(x)dx=(\mathbf{L}%
_{q,\mathfrak{I}_{N}}u,\phi)_{C_{0}^{\infty}(0,b)}= \int_{x_{k}}^{x_{k+1}%
}u(x)\mathbf{L}_{q}\phi(x)dx,
\end{equation}
because $\phi(x_{j})=0$ for $j=1, \dots, N$. From (\ref{auxprop1}) we obtain
\[
\int_{x_{k}}^{x_{k+1}}u(x)\phi^{\prime\prime}(x)dx= \int_{x_{k}}^{x_{k+1}%
}\{q(x)u(x)-g(x)\}\phi(x)dx.
\]
Set $v(x)=\int_{0}^{x}\int_{0}^{t}\{q(s)u(s)-g(s)\}dsdt$. Hence $v\in
W^{2,1}(x_{j},x_{j+1})$, $v^{\prime\prime}(x)=q(x)u(x)-g(x)$ a.e. $x\in
(x_{j},x_{j+1})$, and we get the equality
\begin{equation}
\label{auxiliareq}\int_{x_{k}}^{x_{k+1}}(u(x)-v(x))\phi^{\prime\prime}(x)dx=0
\quad\forall\phi\in C_{0}^{\infty}(x_{k}, x_{k+1}).
\end{equation}
Equality (\ref{auxiliareq}) implies that $u(x)=v(x)+Ax+B$ a.e. $x\in
(x_{k},x_{k+1})$ for some constants $A$ and $B$ (\cite[pp. 85]{vladimirov}).
In consequence $u\in W^{2,1}(x_{k}, x_{k+1})$ and
\begin{equation}
\label{aux2prop1}-u^{\prime\prime}(x)+q(x)u(x)= g(x) \quad\mbox{a.e. }
x\in(x_{k},x_{k+1}).
\end{equation}
Furthermore, $u\in C[x_{k},x_{k+1}]$, hence $qu\in L_{2}(x_{k},x_{k+1})$ and
then $u^{\prime\prime}=qu-g\in L_{2}(x_{k},x_{k+1})$. In this way
$u|_{(x_{k},x_{k+1})}\in H^{2}(x_{k},x_{k+1})$.

Now take $\varepsilon>0$ and an arbitrary $\phi\in C_{0}^{\infty}%
(\varepsilon,x_{1})$. We have that
\[
(\mathbf{L}_{q,\mathfrak{I}_{N}}u,\phi)_{C_{0}^{\infty}(0,b)}= \int
_{\varepsilon}^{x_{1}}\{-u(x)\phi^{\prime\prime}(x)+q(x)u(x)\phi
(x)\}dx=\int_{\varepsilon}^{x_{1}}g(x)\phi(x)dx.
\]
Applying the same procedure as in the previous case we obtain that $u\in
H^{2}(\varepsilon,x_{1})$ and satisfies Eq. (\ref{aux2prop1}) in the interval
$(\varepsilon,x_{1})$. Since $\varepsilon$ is arbitrary, we conclude that $u$
satisfies (\ref{aux2prop1}) for a.e. $x\in(0,x_{1})$. Since $q,g\in
L_{2}(0,x_{1})$, then $u|_{(0,x_{1})}\in H^{2}(0,x_{1})$ (see \cite[Th.
3.4]{zetl}). The proof for the interval $(x_{N},b)$ is analogous.

Since $u\in C^{1}[x_{k},x_{k+1}]$, $k=0,\dots, N$, the following equality is
valid (see formula (6) from \cite[pp. 100]{kanwal})
\begin{align}
\int_{0}^{b} u(x)\phi^{\prime\prime}(x)dx  &  = \sum_{k=1}^{N}\left\{
u^{\prime}(x_{k}+)-u^{\prime}(x_{k}-) \right\} \phi(x_{k})\label{aux3prop1}\\
&  \quad-\sum_{k=1}^{N}\left\{ u(x_{k}+)-u(x_{k}-) \right\} \phi^{\prime
}(x_{k}) +\int_{0}^{b} u^{\prime\prime}(x)\phi(x)dx, \qquad\forall\phi\in
C_{0}^{\infty}(0,b).\nonumber
\end{align}
Fix $k\in\{1, \cdots, N\}$ arbitrary and take $\varepsilon>0$ small enough
such that $(x_{k}-\varepsilon,x_{k}+\varepsilon)\subset(x_{k-1},x_{k+1})$.
Choose a cut-off function $\psi\in C_{0}^{\infty}(x_{k}-\varepsilon
,x_{k}+\varepsilon)$ satisfying $0\leqslant\psi\leqslant1$ on $(x_{k}%
-\varepsilon, x_{k}+\varepsilon)$ and $\psi(x)=1$ for $x\in(x_{k}%
-\frac{\varepsilon}{3}, x_{k}+\frac{\varepsilon}{3})$. 

\item By statement 1, it is enough to show that $u(x_{k}+)=u(x_{k}-)$.
Set\newline$\phi(x)=(x-x_{k})\psi(x)$, in such a way that  $\phi(x_{k})=0$ and
$\phi^{\prime}(x_{k})=1$. Hence
\[
(\mathbf{L}_{q,\mathfrak{I}_{N}}u, \phi)_{C_{0}^{\infty}(0,b)}= \int
_{x_{k}-\varepsilon}^{x_{k}+\varepsilon}u(x)\mathbf{L}_{q}\phi(x)dx.
\]
By (\ref{aux3prop1}) we have
\[
\int_{x_{k}-\varepsilon}^{x_{k}+\varepsilon}u(x)\phi^{\prime\prime}(x)dx =
u(x_{k}-)-u(x_{k}+)+\int_{x_{k}-\varepsilon}^{x_{k}+\varepsilon}%
u^{\prime\prime}(x)\phi(x)dx,
\]
because $\phi(x_{k})=0$ and $\phi^{\prime}(x_{k})=1$. Since $u$ satisfies
(\ref{aux0prop1}), we have
\[
\int_{x_{k}-\varepsilon}^{x_{k}+\varepsilon}(\mathbf{L}_{q}u(x)-g(x))\phi
(x)dx+u(x_{k}+)-u(x_{k}-)=0.
\]
By statement 1, $\mathbf{L}_{q}u=g$ on both intervals $(x_{k-1},x_{k})$,
$(x_{k}, x_{k+1})$. Then we obtain that $u(x_{k}+)-u(x_{k}-)$=0.

\item Now take $\psi$ as the test function. Hence
\[
(\mathbf{L}_{q,\mathfrak{I}_{N}}u,\psi)_{C_{0}^{\infty}(0,b)}= \int
_{x_{k}-\varepsilon}^{x_{k}+\varepsilon}u(x)\mathbf{L}_{q}\psi(x)dx+\alpha
_{k}u(x_{k}),
\]
because $\operatorname{Supp}(\psi)\subset(x_{k}-\varepsilon,x_{k}%
+\varepsilon)$ and $\psi\equiv1$ on $(x_{k}-\frac{\varepsilon}{3}, x_{k}%
+\frac{\varepsilon}{3})$. On the other hand, by (\ref{aux3prop1}) we obtain
\[
\int_{x_{k}-\varepsilon}^{x_{k}+\varepsilon}u(x)\psi^{\prime\prime}(x)dx =
u^{\prime}(x_{k}+)-u^{\prime}(x_{k}-)+\int_{x_{k}-\varepsilon}^{x_{k}%
+\varepsilon}u^{\prime\prime}(x)\psi(x)dx,
\]
because $\psi^{\prime}(x_{k})=0$. Thus, by (\ref{aux0prop1}) we have
\[
\int_{x_{k}-\varepsilon}^{x_{k}+\varepsilon}(\mathbf{L}_{q}u(x)-g(x))\psi
(x)dx+ u^{\prime}(x_{k}-)-u^{\prime}(x_{k}+)+\alpha_{k}u(x_{k})=0.
\]
Again, by statement 1, we obtain (\ref{jumpderivative}). 
\end{enumerate}

Reciprocally, if $u$ satisfies conditions 1,2 and 3, equality (\ref{aux3prop1}%
) implies (\ref{regulardist}). By condition 1, $\mathbf{L}_{q,\mathfrak{I}%
_{N}}u$ is $L_{2}$-regular.
\end{proof}

\begin{definition}
The $L_{2}$-\textbf{regularization domain} of $\mathbf{L}_{q, \mathfrak{I}%
_{N}}$, denoted by $\mathcal{D}_{2}\left( \mathbf{L}_{q, \mathfrak{I}_{N}%
}\right) $, is the set of all functions $u\in L_{2,loc}(0,b)$ satisfying
conditions 1,2 and 3 of Proposition \ref{propregular}.
\end{definition}

If $u\in L_{2,loc}(0,b)$ is a solution of (\ref{Schrwithdelta}), then
$\mathbf{L}_{q-\lambda,\mathfrak{I}_{N}}u$ equals the regular distribution
zero. Then we have the next characterization.

\begin{corollary}
\label{proppropofsolutions}  A function $u\in L_{2,loc}(0,b)$ is a solution of
Eq. (\ref{Schrwithdelta}) iff $u\in\mathcal{D}_{2}\left( \mathbf{L}_{q,
\mathfrak{I}_{N}}\right) $ and  for each $k=0, \dots, N$, the restriction
$u|_{(x_{k},x_{k+1})}$ is a solution of the regular Schr\"odinger equation
\begin{equation}
\label{schrodingerregular}-y^{\prime\prime}(x)+q(x)y(x)=\lambda y(x)
\quad\mbox{for  } x_{k}<x<x_{k+1}.
\end{equation}

\end{corollary}

\begin{remark}
\label{remarkidealdomain}  Let $f\in\mathcal{D}_{2}\left( \mathbf{L}%
_{q,\mathfrak{I}_{N}}\right) $. Given $g\in C^{1}[0,b]$, we have
\begin{align*}
(fg)^{\prime}(x_{k}+)-(fg)^{\prime}(x_{k}-) &  =f^{\prime}(x_{k}%
+)g(x_{k})+f(x_{k})g^{\prime}(x_{k}+)-f^{\prime}(x_{k}-)g(x_{k})-f(x_{k}%
)g^{\prime}(x_{k}-)\\
&  = \left[ f^{\prime}(x_{k}+)-f^{\prime}(x_{k}-)\right] g(x_{k}) = \alpha_{k}
f(x_{k})g(x_{k})
\end{align*}
for $k=1, \dots, N$. In particular, $fg\in\mathcal{D}_{2}\left( \mathbf{L}%
_{q,\mathfrak{I}_{N}}\right) $ for $g\in H^{2}(0,b)$.
\end{remark}

\begin{remark}
\label{remarkunicityofsol}  Let $u_{0}, u_{1}\in\mathbb{C}$. Consider the
Cauchy problem
\begin{equation}
\label{Cauchyproblem1}%
\begin{cases}
\mathbf{L}_{q, \mathfrak{I}_{N}}u(x) = \lambda u(x), \quad0<x<b,\\
u(0)=u_{0}, \; u^{\prime}(0)= u_{1}.
\end{cases}
\end{equation}
If the solution of the problem exists, it must be unique. It is enough to show
the assertion for $u_{0}=u_{1}=0$. Indeed, if $w$ is a solution of such
problem, by Corollary \ref{proppropofsolutions}, $w$ is a solution of
(\ref{schrodingerregular}) on $(0, x_{1})$ satisfying $w(0)=w^{\prime}(0)=0$.
Hence $w\equiv0$ on $[0,x_{1}]$. By the continuity of $w$ and condition
(\ref{jumpderivative}), we have $w(x_{1})=w^{\prime}(x_{1}-)=0$. Hence $w$ is
a solution of (\ref{schrodingerregular}) satisfying these homogeneous
conditions. Thus, $w\equiv0$ on $[x_{1}, x_{2}]$. By continuing the process
until the points $x_{k}$ are exhausted, we arrive at the solution $w\equiv0$
on the whole segment $[0,b]$.

The uniqueness of the Cauchy problem with conditions $u(b)=u_{0}$, $u^{\prime
}(b)=u_{1}$ is proved in a similar way.
\end{remark}

\begin{remark}
Suppose that $u_{0}=u_{0}(\lambda)$ and $u_{1}=u_{1}(\lambda)$ are entire
functions of $\lambda$ and denote by $u(\lambda,x)$ the corresponding unique
solution of (\ref{Cauchyproblem1}). Since $u$ is the solution of the Cauchy
problem $\mathbf{L}_{q}u=\lambda u$ on $(0,x_{1})$ with the initial conditions
$u(\lambda,0)=u_{1}(\lambda)$, $u^{\prime}(\lambda,0)=u_{1}(\lambda)$, both
$u(\lambda,x)$ and $u^{\prime}(\lambda,x+)$ are entire functions for any
$x\in[0,x_{1}]$ (this is a consequence of \cite[Th. 3.9]{zetl} and \cite[Th.
7]{blancarte}). Hence $u^{\prime}(\lambda,x_{1}-)=u^{\prime}(\lambda
,x_{1}+)-\alpha_{1}u(\lambda,x_{1})$ is entire in $\lambda$. Since $u$ is the
solution of the Cauchy problem $\mathbf{L}_{q}u=\lambda u$ on $(x_{1},x_{2})$
with initial conditions $u(\lambda,x_{1})$ and $u^{\prime}(\lambda,x_{1}+)$,
we have that $u(\lambda,x)$ and $u^{\prime}(\lambda,x+)$ are entire functions
for $x\in[x_{1},x_{2}]$. By continuing the process we prove this assertion for
all $x\in[0,b]$.
\end{remark}

\section{Closed form solution}

In what follows, denote the square root of $\lambda$ by $\rho$, so
$\lambda=\rho^{2}$, $\rho\in\mathbb{C}$. For each $k\in\{1, \cdots, N\}$ let
$\widehat{s}_{k}(\rho,x)$ be the unique solution of the Cauchy problem
\begin{equation}%
\begin{cases}
-\widehat{s}_{k}^{\prime\prime}(\rho,x)+q(x+x_{k})\widehat{s}_{k}(\rho
,x)=\rho^{2}\widehat{s}_{k}(\rho, x) \quad\mbox{ for   } 0<x<b-x_{k},\\
\widehat{s}_{k}(\rho,0)=0, \; \widehat{s}_{k}^{\prime}(\rho, 0)=1.
\end{cases}
\end{equation}
In this way, $\widehat{s}_{k}(\rho, x-x_{k})$ is a solution of $\mathbf{L}%
_{q}u=\rho^{2} u$ on $(x_{k},b)$ with initial conditions $u(x_{k})=0$,
$u^{\prime}(x_{k})=1$. According to \cite[Ch. 3, Sec. 6.3]{vladimirov},
$(\mathbf{L}_{q}-\rho^{2})\left( H(x-x_{k})\widehat{s}_{k}(\rho,
x-x_{k})\right) =-\delta(x-x_{k})$ for $x_{k}<x<b$. \newline

We denote by $\mathcal{J}_{N}$ the set of finite sequences $J=(j_{1}, \dots,
j_{l})$ with $1<l\leqslant N$, $\{j_{1}, \dots, j_{l}\}\subset\{1, \dots, N\}$
and $j_{1}<\cdots<j_{l}$. Given $J\in\mathcal{J}_{N}$, the length of $J$ is
denoted by $|J|$ and we define $\alpha_{J}:= \alpha_{j_{1}}\cdots
\alpha_{j_{|J|}}$.

\begin{theorem}
\label{TheoremSolCauchy}  Given $u_{0}, u_{1}\in\mathbb{C}$, the unique
solution $\displaystyle u_{\mathfrak{I}_{N}}\in\mathcal{D}_{2}\left(
\mathbf{L}_{q,\mathfrak{I}_{N}}\right) $ of the Cauchy problem
(\ref{Cauchyproblem1}) has the form
\begin{align}
u_{\mathfrak{I}_{N}}(\rho,x)  &  = \widetilde{u}(\rho,x)+ \sum_{k=1}^{N}%
\alpha_{k}\widetilde{u}(\rho,x_{k})H(x-x_{k})\widehat{s}_{k}(\rho
,x-x_{k})\nonumber\\
& \;\;\;\; +\sum_{J\in\mathcal{J}_{N}}\alpha_{J} H(x-x_{j_{|J|}})
\widetilde{u}(\rho,x_{j_{1}}) \left( \prod_{l=1}^{|J|-1}\widehat{s}_{j_{l}%
}(\rho,x_{j_{l+1}}-x_{j_{l}})\right) \widehat{s}_{j_{|J|}}(\rho, x-x_{j_{|J|}%
}),\label{generalsolcauchy}%
\end{align}
where $\widetilde{u}(\rho,x)$ is the unique solution of the regular
Schr\"odinger equation
\begin{equation}
\label{regularSch}\mathbf{L}_{q}\widetilde{u}(\rho,x)= \rho^{2}\widetilde
{u}(\rho,x), \quad0<x<b,
\end{equation}
satisfying the initial conditions $\widetilde{u}(\rho,0)=u_{1}, \;
\widetilde{u}^{\prime}(\rho,0)=u_{1}$. 
\end{theorem}

\begin{proof}
The proof is by induction on $N$. For $N=1$, the proposed solution has the
form
\[
u_{\mathfrak{I}_{1}}(\rho,x)=\widetilde{u}(\rho,x)+\alpha_{1}H(x-x_{1}%
)\widetilde{u}(\rho,x_{1})\widehat{s}_{1}(\rho,x-x_{1}).
\]
Note that $u_{\mathfrak{I}_{1}}(\rho,x)$ is continuous, and $u_{\mathfrak{I}%
_{1}}(\rho,x_{1})=\widetilde{u}(\rho,x_{1})$. Hence
\[
(\mathbf{L}_{q}-\rho^{2})u_{\mathfrak{I}_{1}}(\rho,x)= \alpha_{1}\widetilde
{u}(\rho,x_{1})(\mathbf{L}_{q}-\rho^{2})\left( H(x-x_{1})\widehat{s}_{1}%
(\rho,x-x_{1})\right) = -\alpha_{1}\widetilde{u}(\rho,x_{1})\delta(x-x_{1}),
\]
that is, $u_{\mathfrak{I}_{1}}(\rho,x)$ is a solution of (\ref{Schrwithdelta})
with $N=1$. Suppose the result is valid for $N$. Let $u_{\mathfrak{I}_{N+1}%
}(\rho,x)$ be the proposed solution given by formula (\ref{generalsolcauchy}).
It is clear that $u_{\mathfrak{I}_{N+1}}(\rho, \cdot)|_{(x_{k},x_{k+1})}\in
H^{2}(x_{k},x_{k+1})$, $k=0, \cdots, N$, $u_{\mathfrak{I}_{N+1}}(\rho,x)$ is a
solution of (\ref{schrodingerregular}) on each interval $(x_{k},x_{k+1})$,
$k=0, \dots, N+1$, and $u_{\mathfrak{I}_{N+1}}^{(j)}(\rho,0)= \widetilde
{u}^{(j)}(\rho,0)=u_{j}$, $j=0,1$. Furthermore, we can write
\[
u_{\mathfrak{I}_{N+1}}(\rho,x)=u_{\mathfrak{I}_{N}}(\rho,x)+H(x-x_{N+1}%
)f_{N}(\rho,x),
\]
where $\mathfrak{I}_{N}= \mathfrak{I}_{N+1}\setminus\{(x_{N+1}, \alpha
_{N+1})\}$, $u_{\mathfrak{J}_{N}}(\rho,x)$ is the proposed solution for the
interactions $\mathfrak{I}_{N}$, and the function $f_{N}(\rho,x)$ is given by
\begin{align*}
f_{N}(\rho,x)  &  = \alpha_{N+1}\widetilde{u}(\rho,x_{N+1})\widehat{s}%
_{N+1}(x-x_{N+1})\\
&  \quad+\sum_{\overset{J\in\mathcal{J}_{N+1}}{j_{|J|}=N+1}}\alpha_{J}
\widetilde{u}(\rho,x_{j_{1}}) \left( \prod_{l=1}^{|J|-1}\widehat{s}_{j_{l}%
}(\rho,x_{j_{l+1}}-x_{j_{l}})\right) \widehat{s}_{N+1}(\rho, x-x_{N+1}),
\end{align*}
where the sum is taken over all the sequences $J=(j_{1}, \dots, j_{|J|}%
)\in\mathcal{J}_{N}$ with $j_{|J|}=N+1$. From this representation we obtain
$u_{\mathfrak{I}_{N+1}}(\rho,x_{N+1}\pm)=u_{\mathfrak{I}_{N}}(\rho,x_{N+1})$
and hence $u_{\mathfrak{I}_{N+1}}\in AC[0,b]$. By the induction hypothesis,
$u_{\mathfrak{I}_{N}}(\rho,x)$ is the solution of (\ref{Schrwithdelta}) for
$N$, then in order to show that $u_{\mathfrak{I}_{N+1}}(\rho,x)$ is the
solution for $N+1$ it is enough to show that $(\mathbf{L}_{q}-\rho^{2})\hat
{f}_{N}(\rho,x)=-\alpha_{N}u_{N}(x_{N+1})\delta(x-x_{N+1})$, where $\hat
{f}_{N}(\rho,x)=H(x-x_{N+1})f_{N}(\rho,x)$. Indeed, we have
\begin{align*}
(\rho^{2}-\mathbf{L}_{q})\hat{f}_{N}(\rho,x)  &  = \alpha_{N+1}\widetilde
{u}(\rho,x_{N+1})\delta(x-x_{N+1})+\\
&  \quad\; +\sum_{\overset{J\in\mathcal{J}_{N+1}}{j_{|J|}=N+1}}\alpha_{J}
\widetilde{u}(\rho,x_{j_{1}}) \left( \prod_{l=1}^{|J|-1}\widehat{s}_{j_{l}%
}(\rho,x_{j_{l+1}}-x_{j_{l}})\right) \delta(x-x_{N+1})\\
& = \alpha_{N+1}\delta(x-x_{N+1})\Bigg{[} \widetilde{u}(\rho,x_{N+1}%
)+\sum_{k=1}^{N}\alpha_{k}\widetilde{u}(\rho,x_{N+1})\widehat{s}_{k}(\rho,
x_{N+1}-x_{k})\\
&  \quad\;\;\,+ \sum_{J\in\mathcal{J}_{N}}\alpha_{J}\widetilde{u}%
(\rho,x_{j_{1}})\left( \prod_{l=1}^{|J|-1}\widehat{s}_{j_{l}}(\rho,x_{j_{l+1}%
}-x_{j_{l}})\right) \widehat{s}_{j_{|J|}}(\rho, x_{N+1}-x_{j_{|J|}})
\Bigg{]}\\
& = \alpha_{N+1}u_{\mathfrak{I}_{N}}(\rho,x_{N+1})\delta(x-x_{N+1}%
)=\alpha_{N+1}u_{\mathfrak{I}_{N+1}}(\rho,x_{N+1})\delta(x-x_{N+1}),
\end{align*}
where the second equality is due to the fact that
\[
\{J\in\mathcal{J}_{N+1}\,|\, j_{|J|}=N+1 \} = \{(J^{\prime},N+1)\, |\,
J^{\prime}\in\mathcal{J}_{N} \}\cup\{(j,N+1) \}_{j=1}^{N}.
\]
Hence $u_{\mathfrak{I}_{N+1}}(\rho,x)$ is the solution of the Cauchy problem. 
\end{proof}

\begin{example}
Consider the case $q\equiv0$. Denote by $e_{\mathfrak{I}_{N}}^{0}(\rho,x)$ the
unique solution of
\begin{equation}
\label{Deltadiracpotentialexample}-y^{\prime\prime}+\left( \sum_{k=1}%
^{N}\alpha_{k} \delta(x-x_{k})\right) y =\rho^{2} y, \quad0<x<b,
\end{equation}
satisfying $e_{\mathfrak{I}_{N}}^{0}(\rho, 0)=1$, $e_{\mathfrak{I}_{N}}%
^{0}(\rho, 0)=i\rho$. In this case we have $\widehat{s}_{k}(\rho,x)=\frac
{\sin(\rho x)}{\rho}$ for $k=1,\dots, N$. Hence, according to Theorem
\ref{TheoremSolCauchy}, the solution $e_{\mathfrak{I}_{N}}^{0}(\rho,x)$ has
the form
\begin{align}
e_{\mathfrak{I}_{N}}^{0}(\rho,x)  &  =e^{i\rho x}+\sum_{k=1}^{N}\alpha
_{k}e^{i\rho x_{k}}H(x-x_{k})\frac{\sin(\rho(x-x_{k}))}{\rho}\nonumber\\
&  \quad+\sum_{J\in\mathcal{J}_{N}}\alpha_{J} H(x-x_{j_{|J|}})e^{i\rho
x_{j_{1}}}\left( \prod_{l=1}^{|J|-1}\frac{\sin(\rho(x_{j_{l+1}}-x_{j_{l}}%
))}{\rho}\right) \frac{\sin(\rho(x-x_{j_{|J|}}))}{\rho}.\label{SolDiracDeltae}%
\end{align}

\end{example}

\section{Transmutation operators}

\subsection{Construction of the integral transmutation kernel}

Let $h\in\mathbb{C}$. Denote by $\widetilde{e}_{h}(\rho,x)$ the unique
solution of Eq. (\ref{regularSch}) satisfying $\widetilde{e}_{h}(\rho,0)=1$,
$\widetilde{e}^{\prime}_{h}(\rho, 0)=i\rho+h$. Hence the unique solution
$e_{\mathfrak{I}_{N}}^{h}(\rho,x)$ of Eq. (\ref{Schrwithdelta}) satisfying
$e_{\mathfrak{I}_{N}}^{h}(\rho,0)=1$, $(e_{\mathfrak{I}_{N}}^{h})^{\prime
}(\rho,0)=i\rho+h$ is given by
\begin{align}
e_{\mathfrak{I}_{N}}^{h}(\rho,x)  &  =\widetilde{e}_{h}(\rho,x)+\sum_{k=1}%
^{N}\alpha_{k}\widetilde{e}_{h}(\rho, x_{k})H(x-x_{k})\widehat{s}_{k}(\rho,
x-x_{k})\label{SoleGral}\\
&  \quad+\sum_{J\in\mathcal{J}_{N}}\alpha_{J} H(x-x_{j_{|J|}})\widetilde
{e}_{h}(\rho,x_{j_{1}})\left( \prod_{l=1}^{|J|-1}\widehat{s}_{j_{l}}%
(\rho,x_{j_{l+1}}-x_{j_{l}})\right) \widehat{s}_{j_{|J|}}(\rho,x-x_{j_{|J|}%
}).\nonumber
\end{align}

It is known that there exists a kernel $\widetilde{K}^{h}\in C(\overline
{\Omega})\cap H^{1}(\Omega)$, where $\Omega=\{(x,t)\in\mathbb{R}^{2}\, |\,
0<x<b, |t|<x\}$, such that $\widetilde{K}^{h}(x,x)=\frac{h}{2}+\frac{1}{2}%
\int_{0}^{x}q(s)ds$, $\widetilde{K}^{h}(x,-x)=\frac{h}{2}$ and
\begin{equation}
\label{transm1}\widetilde{e}_{h}(\rho,x)=e^{i\rho x}+\int_{-x}^{x}%
\widetilde{K}^{h}(x,t)e^{i\rho t}dt
\end{equation}
(see, e.g., \cite{levitan, marchenko}). Actually, $\widetilde{K}^{h}%
(x,\cdot)\in L_{2}(-x,x)$ and it can be extended (as a function of $t$) to a
function in $L_{2}(\mathbb{R})$ with a support in $[-x,x]$. For each $k\in\{1,
\dots, N\}$ there exists a kernel $\widehat{H}_{k}\in C(\overline{\Omega_{k}%
})\cap H^{1}(\Omega_{k})$ with $\Omega_{k}=\{(x,t)\in\mathbb{R}^{2}\,|\,
0<x<b-x_{k}, \; |t|\leqslant x\}$, and $\widehat{H}_{k}(x,x)=\frac{1}{2}%
\int_{x_{k}}^{x+x_{k}}q(s)ds$, $\widehat{H}_{k}(x,-x)=0$, such that
\begin{equation}
\label{representationsinegeneral1}\widehat{s}_{k}(\rho,x)=\frac{\sin(\rho
x)}{\rho}+\int_{0}^{x}\widehat{H}_{k}(x,t)\frac{\sin(\rho t)}{\rho}dt
\end{equation}
(see \cite[Ch. 1]{yurko}). From this we obtain the representation
\begin{equation}
\label{representationsinegeneral2}\widehat{s}_{k}(\rho,x-x_{k})=\frac
{\sin(\rho(x-x_{k}))}{\rho}+\int_{0}^{x-x_{k}}\widehat{H}_{k}(x-x_{k}%
,t)\frac{\sin(\rho t)}{\rho}dt=\int_{-(x-x_{k})}^{x-x_{k}}\widetilde{K}%
_{k}(x,t)e^{i\rho t}dt,
\end{equation}
where
\begin{equation}
\label{kernelauxiliarsine}\widetilde{K}_{k}(x,t)=\frac{1}{2}\chi_{x-x_{k}%
}(t)+\displaystyle \frac{1}{2}\int_{|t|}^{x-x_{k}}\widehat{H}_{k}%
(x-x_{k},s)ds.
\end{equation}

We denote the Fourier transform of a function $f\in L_{1}(\mathbb{R})$ by
$\mathcal{F}f(\rho)=\int_{\mathbb{R}}f(t)e^{i\rho t}dt$ and the convolution of
$f$ with a function $g\in L_{1}(\mathbb{R})$ by $f\ast g(t)= \int_{\mathbb{R}%
}f(t-s)g(s)ds$. We recall that $\mathcal{F}(f\ast g)(\rho)=\mathcal{F}%
f(\rho)\cdot\mathcal{F}g(\rho)$. Given $f_{1}, \dots, f_{M} \in
L_{2}(\mathbb{R})$ with compact support, we denote their convolution product
by $\left( \prod_{l=1}^{M}\right) ^{\ast}f_{l}(t):= (f_{1}\ast\cdots\ast
f_{M})(t)$. For the kernels $\widetilde{K}^{h}(x,t), \widetilde{K}_{k}(x,t)$,
the operations $\mathcal{F}$ and $\ast$ will be applied with respect to the
variable $t$.

\begin{lemma}
\label{lemaconv}  Let $A,B>0$. If $f\in C[-A,A]$ and $g\in C[-B,B]$, then
$(\chi_{A} f)\ast(\chi_{B} g)\in C(\mathbb{R})$ with $\operatorname{Supp}%
\left( (\chi_{A} f)\ast(\chi_{B} g)\right) \subset[-(A+B),A+B]$.
\end{lemma}

\begin{proof}
The assertion $\operatorname{Supp}\left( (\chi_{A} f)\ast(\chi_{B} g)\right)
\subset[-(A+B),A+B]$ is due to \cite[Prop. 4.18]{brezis}. Since $(\chi_{A}
f)\in L_{1}(\mathbb{R})$ and $(\chi_{B} g)\in L_{\infty}(\mathbb{R})$, it
follows from \cite[Prop. 8.8]{folland} that $(\chi_{A} f)\ast(\chi_{B} g)\in
C(\mathbb{R})$.
\end{proof}

\begin{theorem}
\label{thoremtransmoperator}  There exists a kernel $K_{\mathfrak{I}_{N}}%
^{h}(x,t)$ defined on $\Omega$ such that
\begin{equation}
\label{transmutationgeneral}e_{\mathfrak{I}_{N}}^{h}(\rho,x)=e^{i\rho x}%
+\int_{-x}^{x}K_{\mathfrak{I}_{N}}^{h}(x,t)e^{i\rho t}dt.
\end{equation}
For any $0<x\leqslant b$, $K_{\mathfrak{J}_{N}}^{h}(x,t)$ is piecewise
absolutely continuous with respect to the variable $t\in[-x,x]$ and satisfies
$K_{\mathfrak{I}_{N}}^{h}(x,\cdot)\in L_{2}(-x,x)$. Furthermore,
$K_{\mathfrak{I}_{N}}^{h}\in L_{\infty}(\Omega)$.
\end{theorem}

\begin{proof}
Susbtitution of formulas (\ref{transm1}) and (\ref{representationsinegeneral2}%
) in (\ref{SoleGral}) leads to the equality
\begin{align*}
& e_{\mathfrak{I}_{N}}^{h}(\rho,x)= e^{i\rho x}+\int_{-x}^{x}\widetilde{K}%
^{h}(x,t)e^{i\rho t}dt+\\
& +\sum_{k=1}^{N}\alpha_{k}H(x-x_{k})\left( e^{i\rho x_{k}}+\int
\limits_{-x_{k}}^{x_{k}}\widetilde{K}^{h}(x_{k},t)e^{i\rho t}dt\right) \left(
\int\limits_{-(x-x_{k})}^{x-x_{k}}\widetilde{K}_{k}(x,t)e^{i\rho t}dt\right)
\\
& +\sum_{J\in\mathcal{J}_{N}}\alpha_{J}H(x-x_{j_{|J|}})\Bigg[\left( e^{i\rho
x_{j_{1}}}+\int\limits_{-x_{j_{1}}}^{x_{j_{1}}}\widetilde{K}^{h}(x_{j_{1}%
},t)e^{i\rho t}dt\right) \left( \prod_{l=1}^{|J|-1}\int\limits_{-(x_{j_{l+1}%
}-x_{j_{l}})}^{x_{j_{l+1}}-x_{j_{l}}}\widetilde{K}_{k}(x_{j_{l+1}},t)e^{i\rho
t}dt\right) \\
&  \qquad\qquad\qquad\qquad\cdot\int\limits_{-(x-x_{j_{|J|}})}^{x-x_{j_{|J|}}%
}\widetilde{K}_{k}(x,t)e^{i\rho t}dt\Bigg]
\end{align*}
Note that
\begin{align*}
\prod_{l=1}^{|J|-1}\int\limits_{-(x_{j_{l+1}}-x_{j_{l}})}^{x_{j_{l+1}%
}-x_{j_{l}}}\widetilde{K}_{k}(x_{j_{l+1}},t)e^{i\rho t}dt  &  = \mathcal{F}%
\left\{ \left( \prod_{l=1}^{|J|-1}\right) ^{\ast}\left( \chi_{x_{j_{l+1}%
}-x_{j_{l}}}(t)\widetilde{K}_{k}(x_{j_{l+1}},t)\right) \right\} .
\end{align*}
In a similar way, if we denote $I_{A,B}=\left( e^{i\rho A}+\int\limits_{-A}%
^{A}\widetilde{K}^{h}(A,t)e^{i\rho t}dt\right) \left( \int\limits_{-B}%
^{B}\widetilde{K}_{k}(B,t)e^{i\rho t}dt\right) $ with $A,B\in(0,b)$, then
\begin{align*}
I_{A,B} =  &  e^{i\rho A}\int\limits_{-B}^{B}\widetilde{K}_{k}(B,t)e^{i\rho
t}dt+ \mathcal{F}\left( \chi_{A}(t)\widetilde{K}^{h}(A,t)\ast\chi
_{B}(t)\widetilde{K}_{k}(B,t)\right) \\
= &  \mathcal{F}\left(  \chi_{[A-B,B+A]}(t)\widetilde{K}_{k}(B,t-A)+ \chi
_{A}(t)\widetilde{K}^{h}(A,t)\ast\chi_{B}(t)\widetilde{K}_{k}(B,t)\right) .
\end{align*}
Set $R_{N}(\rho,x)=e_{N}(\rho,x)-e^{i\rho x}$. Thus,
\begin{align*}
R_{N}(\rho,x) = &  \mathcal{F}\Bigg[ \chi_{x}(t)\widetilde{K}^{h}(x,t)\\
& +\sum_{k=1}^{N}\alpha_{k}H(x-x_{k}) \left( \chi_{[2x_{k}-x,x]}%
(t)\widetilde{K}_{k}(x,t-x_{k})+ \chi_{x_{k}}(t)\widetilde{K}^{h}(x_{k}%
,t)\ast\chi_{x-x_{k}}(t)\widetilde{K}_{k}(x,t)\right) \\
& +\sum_{J\in\mathcal{J}_{N}}\alpha_{J}H(x-x_{j_{|J|}})\left( \prod
_{l=1}^{|J|-1}\right) ^{\ast}\left( \chi_{x_{j_{l+1}}-x_{j_{l}}}%
(t)\widetilde{K}_{k}(x_{j_{l+1}},t)\right) \\
& \ast\Big(\chi_{[x_{j_{|J|}}+x_{j_{1}}-x, x-(x_{j_{|J|}}-x_{j_{1}}%
)]}(t)\widetilde{K}_{j_{|J|}}(x,t-x_{j_{1}})\\
&  \qquad\; + \chi_{x_{j_{1}}}(t)\widetilde{K}^{h}(x_{j_{1}},t)\ast
\chi_{x-x_{j_{|J|}}}(t)\widetilde{K}_{j_{|J|}}(x,t)\Big)\Bigg]
\end{align*}
According to Lemma \ref{lemaconv}, the support of $\left( \prod_{l=1}%
^{|J|-1}\right) ^{\ast}\left( \chi_{x_{j_{l+1}}-x_{j_{l}}}(t)\widetilde{K}%
_{k}(x_{j_{l+1}},t)\right) $ lies in \newline$[x_{j_{1}}-x_{j_{|J|}},
x_{j_{|J|}}-x_{j_{1}}]$ and $\chi_{ x-(x_{j_{|J|}}-x_{j_{1}})}(t)\widetilde
{K}_{j_{|J|}}(x,t-x_{j_{1}})+ \chi_{x_{j_{1}}}(t)\widetilde{K}^{h}(x_{j_{1}%
},t)\ast\chi_{x-x_{j_{|J|}}}(t)\widetilde{K}_{j_{|J|}}(x,t)$ has its support
in $[x_{j_{|J|}}+x_{j_{1}}-x, x-(x_{j_{|J|}}-x_{j_{1}})]$. Hence the
convolution in the second sum of $R_{N}(\rho,x)$ has its support in $[-x,x]$.
On the other hand, $\chi_{x_{k}}(t)\widetilde{K}^{h}(x_{k},t)\ast\chi
_{x-x_{k}}(t)\widetilde{K}_{k}(x,t)$ has its support in $[-x,x]$, and since
$[2x_{k}-x,x]\subset[-x,x]$, we conclude that $\operatorname{Supp}\left(
\mathcal{F}^{-1}R_{N}(\rho,x)\right) \subset[-x,x]$.

Thus, we obtain (\ref{transmutationgeneral}) with
\begin{align}
K_{\mathfrak{I}_{N}}^{h}(x,t) =  &  \chi_{x}(t)\widetilde{K}^{h}%
(x,t)\nonumber\\
& +\sum_{k=1}^{n}\alpha_{k}H(x-x_{k}) \left( \chi_{[2x_{k}-x,x]}%
(t)\widetilde{K}_{k}(x,t-x_{k})+ \chi_{x_{k}}(t)\widetilde{K}^{h}(x_{k}%
,t)\ast\chi_{x-x_{k}}(t)\widetilde{K}_{k}(x,t)\right) \nonumber\\
& +\sum_{J\in\mathcal{J}_{N}}\alpha_{J}H(x-x_{j_{|J|}})\left( \prod
_{l=1}^{|J|-1}\right) ^{\ast}\left( \chi_{x_{j_{l+1}}-x_{j_{l}}}%
(t)\widetilde{K}_{j_{l}}(x_{j_{l+1}},t)\right) \label{transmkernelgeneral}\\
& \qquad\ast\Big(\chi_{ x-(x_{j_{|J|}}-x_{j_{1}})}(t)\widetilde{K}_{j_{|J|}%
}(x,t-x_{j_{1}}) + \chi_{x_{j_{1}}}(t)\widetilde{K}^{h}(x_{j_{1}},t)\ast
\chi_{x-x_{j_{|J|}}}(t)\widetilde{K}_{j_{|J|}}(x,t)\Big),\nonumber
\end{align}
and $K_{\mathfrak{I}_{N}}(x,\cdot)\in L_{2}(x,-x)$. By formula
(\ref{transmkernelgeneral}) and the definitions of $\widehat{K}^{h}(x,t)$ and
$\widetilde{K}_{k}(x,t)$, $K_{\mathfrak{I}_{N}}(x,t)$ is piecewise absolutely
continuous for $t\in[-x,x]$. Since $\widehat{K}^{h},\widetilde{K}_{k}\in
L_{\infty}(\Omega)$, is clear that $K_{\mathfrak{I}_{N}}^{f}\in L_{\infty
}(\Omega)$.
\end{proof}
\newline

As a consequence of (\ref{transmutationgeneral}), $e_{\mathfrak{I}_{N}}%
^{h}(\rho,x)$ is an entire function of exponential type $x$ on the spectral
parameter $\rho$.

\begin{example}
\label{beginexample1}  Consider (\ref{SolDiracDeltae}) with $N=1$. In this
case the solution $e_{\mathfrak{I}_{1}}^{0}(\rho,x)$ is given by
\[
e_{\mathfrak{I}_{1}}^{0}(\rho,x)= e^{i\rho x}+\alpha_{1}e^{i\rho x_{1}%
}H(x-x_{1})\frac{\sin(\rho(x-x_{1}))}{\rho}.
\]
We have
\[
e^{i\rho x_{1}}\frac{\sin(\rho(x-x_{1}))}{\rho}=\frac{1}{2}\int_{x_{1}%
-x}^{x-x_{1}}e^{i\rho(t+x_{1})}dt= \frac{1}{2}\int_{2x_{1}-x}^{x}e^{i\rho
t}dt.
\]
Hence
\[
e_{\mathfrak{I}_{1}}^{0}(\rho,x)= e^{i\rho x}+\int_{-x}^{x}K_{\mathfrak{I}%
_{1}}^{0}(x,t)e^{i\rho t}dt \quad\mbox{with  }\, K_{\mathfrak{I}_{1}}%
^{0}(x,t)=\frac{\alpha_{1}}{2}H(x-x_{1})\chi_{[2x_{1}-x,x]}(t).
\]

\end{example}

\begin{example}
\label{beginexample2}  Consider again Eq. (\ref{SolDiracDeltae}) but now with
$N=2$. In this case the solution $e_{\mathfrak{I}_{2}}^{0}(\rho,x)$ is given
by
\begin{align*}
e_{\mathfrak{I}_{2}}^{0}(\rho,x) =  &  e^{i\rho x}+\alpha_{1}e^{i\rho x_{1}%
}H(x-x_{1})\frac{\sin(\rho(x-x_{1}))}{\rho}+\alpha_{2}e^{i\rho x_{2}}%
H(x-x_{2})\frac{\sin(\rho(x-x_{2}))}{\rho}\\
&  +\alpha_{1}\alpha_{2}e^{i\rho x_{1}}H(x-x_{2})\frac{\sin(\rho(x_{2}%
-x_{1}))}{\rho}\frac{\sin(\rho(x-x_{2}))}{\rho},
\end{align*}
and the transmutation kernel $K_{\mathfrak{I}_{2}}^{0}(x,t)$ has the form
\begin{align*}
K_{\mathfrak{I}_{2}}^{0}(x,t) &  = \frac{\alpha_{1}H(x-x_{1})}{2}\chi
_{[2x_{1}-x,x]}(t)+\frac{\alpha_{2}H(x-x_{2})}{2}\chi_{[2x_{1}-x,x]}(t)\\
& \qquad+\frac{\alpha_{1}\alpha_{2}H(x-x_{2})}{4}\left( \chi_{x_{2}-x_{1}}%
\ast\chi_{x-x_{2}}\right) (t-x_{1}).
\end{align*}
Direct computation shows that
\begin{multline*}
\chi_{x_{2}-x_{1}}\ast\chi_{x-x_{2}}(t-x_{1})=\\%
\begin{cases}
0, & t\not \in [2x_{1}-x,x],\\
t+x-2x_{1}, & 2x_{1}-x< t< -|2x_{2}-x-x_{1}|+x_{1},\\
x-x_{1}-|2x_{2}-x-x_{1}|, & -|2x_{2}-x-x_{1}|+x_{1}< t<|2x_{2}-x-x_{1}%
|+x_{1}\\
x-t, & |2x_{2}-x-x_{1}|+x_{1}<t<x.
\end{cases}
\end{multline*}
In Figure \ref{levelcurves}, we can see some level curves of the kernel
$K_{\mathfrak{I}_{2}}^{0}(x,t)$ (as a function of $t$), $\mathfrak{I_{2}%
}=\{(0.25,1), (0.75,2)\}$, for some values of $x$. \begin{figure}[h]
\centering
\includegraphics[width=11.cm]{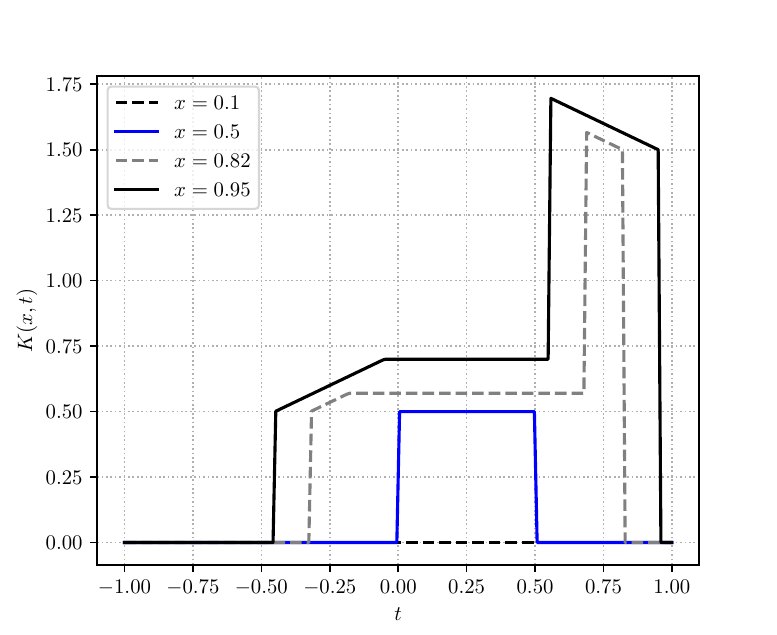} \caption{The graphs of
$K_{\mathfrak{I}_{2}}^{0}(x,t)$, as a function of $t\in[-1,1]$, for some
points $x\in(0,1)$ and $\mathfrak{I_{2}}=\{(0.25,1), (0.75,2)\}$.}%
\label{levelcurves}%
\end{figure}
\end{example}

For the general case we have the following representation for the kernel.

\begin{proposition}
The transmutation kernel $K_{\mathfrak{I}_{N}}^{0}(\rho,x)$ for the solution
$e_{\mathfrak{I}_{N}}^{0}(\rho,x)$ of (\ref{SolDiracDeltae}) is given by
\begin{align}
K_{\mathfrak{I}_{N}}^{0}(x,t)  &  =  \sum_{k=0}^{N}\frac{\alpha_{k}H(x-x_{k}%
)}{2}\chi_{[2x_{k}-x,x]}(t)\nonumber\\
&   +\sum_{J\in\mathcal{J}_{N}}\frac{\alpha_{J}H(x-x_{j_{|J|}})}{2^{|J|}%
}\left( \left( \prod_{l=1}^{|J|-1}\right) ^{\ast}\chi_{x_{j_{l+1}}-x_{j_{l}}%
}(t) \right) \ast\chi_{x-x_{j_{|J|}}}(t-x_{j_{1}})\label{kerneldeltaN}%
\end{align}

\end{proposition}

\begin{proof}
In this case $\widetilde{e}_{0}(\rho,x)=e^{i\rho x}$, $\widehat{s}_{k}%
(\rho,x-x_{k})=\frac{\sin(\rho(x-x_{k}))}{\rho}$, hence $\widetilde{K}%
^{0}(x,t)\equiv0$, $\widetilde{K}_{k}(x,t)=\frac{1}{2}\chi_{x-x_{k}}(t)$.
Substituting these expressions into (\ref{transmkernelgeneral}) and taking
into account that $\chi_{x_{j_{|J|}}+x_{j_{1}}-x, x-(x_{j_{|J|}}-x_{j_{1}}%
)}(t)=\chi_{x-x_{j_{|J|}}}(t-x_{j_{1}})$ we obtain (\ref{kerneldeltaN})
\end{proof}
\newline

Let
\begin{equation}
\label{transmutationoperator1}\mathbf{T}_{\mathfrak{I}_{N}}^{h} u(x):= u(x)+
\int_{-x}^{x} K_{\mathfrak{I}_{N}}^{h}(x,t)u(t)dt.
\end{equation}
By Theorem \ref{thoremtransmoperator}, $\mathbf{T}_{\mathfrak{I}_{N}}^{f}%
\in\mathcal{B}\left( L_{2}(-b,b)\right) $ and
\begin{equation}
\label{transmutationrelation1}e_{\mathfrak{I}_{N}}^{h}(\rho,x)= \mathbf{T}_{
\mathfrak{I}_{N}}^{h} \left[ e^{i\rho x}\right] .
\end{equation}

\subsection{Goursat conditions}

Let us define the function
\begin{equation}
\label{Sigmafunction}\sigma_{\mathfrak{I}_{N}}(x):= \sum_{k=1}^{N}\alpha
_{k}H(x-x_{k}).
\end{equation}
Hence $\sigma_{\mathfrak{I}_{N}}^{\prime}(x)=q_{\delta,\mathfrak{I}_{n}}(x)$
in the distributional sense ( $(\sigma_{\mathfrak{I}_{N}},\phi)_{C_{0}%
^{\infty}(0,b)}=-(q_{\delta,\mathfrak{I}_{N}}, \phi^{\prime})_{C_{0}^{\infty
}(0,b)}$ for all $\phi\in C_{0}^{\infty}(0,b)$). Note that in Examples
\ref{beginexample1} and \ref{beginexample2} we have
\[
K_{\mathfrak{J}_{N}}^{0}(x,x)= \frac{1}{2}\left( \int_{0}^{x}q(s)ds+\sigma
_{\mathfrak{I}_{N}}(x)\right)  \;\,\mbox{   and    }\;\, K_{\mathfrak{I}_{N}%
}^{0}(x,-x)=0 \;\, \mbox{  for } N=1,2.
\]
More generally, the following statement is true.

\begin{proposition}
\label{propGoursat}  The integral transmutation kernel $K_{\mathfrak{I}_{N}%
}^{h}$ satisfies the following Goursat conditions for $x\in[0,b]$
\begin{equation}
K_{\mathfrak{J}_{N}}^{h}(x,x)= \frac{1}{2}\left( h+\int_{0}^{x}q(s)ds+\sigma
_{\mathfrak{I}_{N}}(x)\right) \qquad\mbox{   and    }\qquad K_{\mathfrak{I}%
_{N}}^{h}(x,-x)=\frac{h}{2}.
\end{equation}

\end{proposition}

\begin{proof}
Fix $x\in[0,b]$ and take $\xi\in\{-x,x\}$. By formula
(\ref{transmkernelgeneral}) we can write
\[
K_{\mathfrak{I}_{N}}^{h}(x,\xi)= \widetilde{K}^{h}(x,\xi)+\sum_{k=1}^{N}%
\alpha_{k}H(x-x_{k})\chi_{[2x_{k}-x,x]}(\xi)\widetilde{K}_{k}(x,\xi
-x_{k})+F(x,\xi),
\]
where
\begin{align*}
F(x,t) =  &  \sum_{k=1}^{n}\alpha_{k}H(x-x_{k}) \chi_{x_{k}}(t)\widetilde
{K}^{h}(x_{k},t)\ast\chi_{x-x_{k}}(t)\widetilde{K}_{k}(x,t)\\
& +\sum_{J\in\mathcal{J}_{N}}\alpha_{J}H(x-x_{j_{|J|}})\left( \prod
_{l=1}^{|J|-1}\right) ^{\ast}\left( \chi_{x_{j_{l+1}}-x_{j_{l}}}%
(t)\widetilde{K}_{j_{l}}(x_{j_{l+1}},t)\right) \\
& \qquad\ast\Big(\chi_{ x-(x_{j_{|J|}}-x_{j_{1}})}(t)\widetilde{K}_{j_{|J|}%
}(x,t-x_{j_{1}}) + \chi_{x_{j_{1}}}(t)\widetilde{K}^{h}(x_{j_{1}},t)\ast
\chi_{x-x_{j_{|J|}}}(t)\widetilde{K}_{j_{|J|}}(x,t)\Big).
\end{align*}

In the proof of Theorem \ref{thoremtransmoperator} we obtain that
$\operatorname{Supp}(F(x,t))\subset[-x,x]$. Since $\widetilde{K}^{h}(x_{j},t)$
and $\widetilde{K}_{k}(x_{j},t)$ are continuous with respect to $t$ in the
intervals $[-x_{j},x_{j}]$ and $[x_{k}-x_{j},x_{j}-x_{k}]$ respectively for
$j=1, \dots, N$, $k\leqslant j$, by Lemma \ref{lemaconv} the function $F(x,t)$
is continuous for all $t\in\mathbb{R}$. Thus $F(x,\xi)=0$. For the case
$\xi=x$, we have that $\widetilde{K}^{h}(x,x)=\frac{h}{2}+\frac{1}{2}\int
_{0}^{x}q(s)ds$, $\chi_{[2x_{k}-x,x]}(x)=1$ and
\[
\widetilde{K}_{k}(x,x-x_{k})=\frac{1}{2}\chi_{x-x_{k}}(x-x_{k}%
)+\displaystyle \frac{1}{2}\int_{|x-x_{k}|}^{x-x_{k}}\widehat{H}_{k}%
(x-x_{k},s)ds= \frac{1}{2}
\]
(we assume that $x\geqslant x_{k}$ in order to have $H(x-x_{k})=1$).
Thus\newline$K_{\mathfrak{I}_{N}}^{h}(x,x)=\frac{1}{2}\left( h+\int_{0}%
^{x}q(s)ds+\sigma_{\mathfrak{I}_{N}}(x)\right) $. For the case $\xi=-x$,
$\widetilde{K}^{h}(x,-x)=\frac{h}{2}$ and\newline$\chi_{[2x_{k}-x,x]}(-x)=0$.
Hence $K_{\mathfrak{I}_{N}}^{h}(x,x)=\frac{h}{2}$.
\end{proof}

\begin{remark}
\label{remarkantiderivativegousart} According to Proposition \ref{propGoursat}%
, $2K_{\mathfrak{I}_{N}}^{h}(x,x)$ is a (distributional) antiderivative of the
potential $q(x)+q_{\delta,\mathfrak{I}_{N}}(x)$.
\end{remark}

\subsection{The transmuted Cosine and Sine solutions}

Let $c_{\mathfrak{I}_{N}}^{h}(\rho,x)$ and $s_{\mathfrak{I}_{N}}(\rho,x)$ be
the solutions of Eq. (\ref{Schrwithdelta}) satisfying the initial conditions
\begin{align}
c_{\mathfrak{I}_{N}}^{h}(\rho, 0)= 1,  &\quad  (c_{\mathfrak{I}_{N}}^{h})^{\prime
}(\rho, 0)=h,\label{cosinegralinitialconds}\\
s_{\mathfrak{I}_{N}}(\rho, 0)=0,  & \quad s_{\mathfrak{I}_{N}}^{\prime}(\rho,0)=
1.\label{sinegralinitialconds}%
\end{align}
Note that $c_{\mathfrak{I}_{N}}^{h}(\rho,x)=\frac{e_{\mathfrak{I}_{N}}%
^{h}(\rho,x)+e_{\mathfrak{I}_{N}}^{h}(-\rho,x)}{2}$ and $s_{\mathfrak{I}_{N}%
}(\rho,x)= \frac{e_{\mathfrak{I}_{N}}^{h}(\rho,x)-e_{\mathfrak{I}_{N}}%
^{h}(-\rho,x)}{2i\rho}$.

\begin{remark}
\label{remarksinecosinefunda}  By Corollary \ref{proppropofsolutions},
$c_{\mathfrak{I}_{N}}^{h}(\rho,\cdot), s_{\mathfrak{I}_{N}}(\rho,\cdot)\in
AC[0,b]$ and both functions are solutions of Eq. (\ref{schrodingerregular}) on
$[0,x_{1}]$, hence their Wronskian is constant for $x\in[0,x_{1}]$ and
\begin{align*}
1 &  = W\left[ c_{\mathfrak{I}_{N}}^{h}(\rho,x), s_{\mathfrak{I}_{N}}%
(\rho,x)\right] (0)=W\left[ c_{\mathfrak{I}_{N}}^{h}(\rho,x), s_{\mathfrak{I}%
_{N}}(\rho,x)\right] (x_{1}-) =
\begin{vmatrix}
c_{\mathfrak{I}_{N}}^{h}(\rho,x_{1}) & s_{\mathfrak{I}_{N}}(\rho,x_{1})\\
(c_{\mathfrak{I}_{N}}^{h})^{\prime}(\rho,x_{1}-) & s^{\prime}_{\mathfrak{I}%
_{N}}(\rho,x_{1}-)
\end{vmatrix}
\\
& =
\begin{vmatrix}
c_{\mathfrak{I}_{N}}^{h}(\rho,x_{1}) & s_{\mathfrak{I}_{N}}(\rho,x_{1})\\
(c_{\mathfrak{I}_{N}}^{h})^{\prime}(\rho,x_{1}+)-\alpha_{1} c_{\mathfrak{I}%
_{N}}^{h}(\rho,x_{1}) & s^{\prime}_{\mathfrak{I}_{N}}(\rho,x_{1}+)-\alpha_{1}
s_{\mathfrak{I}_{N}}(\rho,x_{1})
\end{vmatrix}
\\
& =
\begin{vmatrix}
c_{\mathfrak{I}_{N}}^{h}(\rho,x_{1}) & s_{\mathfrak{I}_{N}}(\rho,x_{1})\\
(c_{\mathfrak{I}_{N}}^{h})^{\prime}(\rho,x_{1}+) & s^{\prime}_{\mathfrak{I}%
_{N}}(\rho,x_{1}+)
\end{vmatrix}
= W\left[ c_{\mathfrak{I}_{N}}^{h}(\rho,x), s_{\mathfrak{I}_{N}}%
(\rho,x)\right] (x_{1}+)
\end{align*}
(the equality in the second line is due to (\ref{jumpderivative})). Since
$c_{\mathfrak{I}_{N}}^{h}(\rho,x), s_{\mathfrak{I}_{N}}(\rho,x)$ are solutions
of (\ref{schrodingerregular}) on $[x_{1},x_{2}]$, then $W\left[
C_{\mathfrak{I}_{N}}^{h}(\rho,x), s_{\mathfrak{I}_{N}}(\rho,x)\right] $ is
constant for $x\in[x_{1},x_{2}]$. Thus, \newline$W\left[ c_{\mathfrak{I}_{N}%
}^{h}(\rho,x), s_{\mathfrak{I}_{N}}(\rho,x)\right] (x)=1$ for all
$x\in[0,x_{2}]$. Continuing the process we obtain that the Wronskian equals
one in the whole segment $[0,b]$. Thus, $c_{\mathfrak{I}_{N}}^{h}(\rho,x),
s_{\mathfrak{I}_{N}}(\rho,x)$ are linearly independent. Finally, if $u$ is a
solution of (\ref{Schrwithdelta}), by Remark \ref{remarkunicityofsol}, $u$ can
be written as $u(x)=u(0)c_{\mathfrak{I}_{N}}^{h}(\rho,x)+u^{\prime
}(0)s_{\mathfrak{I}_{N}}(\rho,x)$. In this way, $\left\{ c_{\mathfrak{I}_{N}%
}^{h}(\rho,x), s_{\mathfrak{I}_{N}}(\rho,x)\right\} $ is a fundamental set of
solutions for (\ref{Schrwithdelta}).
\end{remark}

Similarly to the case of the regular Eq. (\ref{regularSch}) (see \cite[Ch.
1]{marchenko}), from (\ref{transmutationgeneral}) we obtain the following representations.

\begin{proposition}
The solutions $c_{\mathfrak{I}_{N}}^{h}(\rho,x)$ and $s_{\mathfrak{I}_{N}%
}(\rho,x)$ admit the following integral representations
\begin{align}
c_{\mathfrak{I}_{N}}^{h}(\rho, x) &  =  \cos(\rho x)+\int_{0}^{x}
G_{\mathfrak{I}_{N}}^{h}(x,t)\cos(\rho t)dt,\label{transmcosinegral}\\
s_{\mathfrak{I}_{N}}(\rho,x)  &  =  \frac{\sin(\rho x)}{\rho}+\int_{0}%
^{x}S_{\mathfrak{I}_{N}}(x,t)\frac{\sin(\rho t)}{\rho}%
dt,\label{transmsinegral}%
\end{align}
where
\begin{align}
G_{\mathfrak{I}_{N}}^{h}(x,t)  &  =  K_{\mathfrak{I}_{N}}^{h}%
(x,t)+K_{\mathfrak{I}_{N}}^{h}(x,-t),\label{cosineintegralkernelgral}\\
S_{\mathfrak{I}_{N}}(x,t)  &  =  K_{\mathfrak{I}_{N}}^{h}(x,t)-K_{\mathfrak{I}%
_{N}}^{h}(x,-t).\label{sineintegralkernelgral}%
\end{align}

\end{proposition}

\begin{remark}
\label{remarkGoursatcosinesine}  By Proposition \ref{propGoursat}, the cosine
and sine integral transmutation kernels satisfy the conditions
\begin{equation}
G_{\mathfrak{I}_{N}}^{h}(x,x)= h+\frac{1}{2}\left( \int_{0}^{x}q(s)ds+\sigma
_{\mathfrak{I}_{N}}(x)\right) ,
\end{equation}

\begin{equation}
S_{\mathfrak{I}_{N}}(x,x)= \frac{1}{2}\left( \int_{0}^{x}q(s)ds+\sigma
_{\mathfrak{I}_{N}}(x)\right)  \quad\mbox{and }\;\; S_{\mathfrak{I}_{N}%
}(x,0)=0.
\end{equation}
Introducing the cosine and sine transmutation operators
\begin{equation}
\label{cosineandsinetransop}\mathbf{T}_{\mathfrak{I}_{N},h}^{C}u(x)=
u(x)+\int_{0}^{x}G_{\mathfrak{I}_{N}}^{h}(x,t)u(t)dt, \quad\mathbf{T}%
_{\mathfrak{I}_{N}}^{S}u(x)= u(x)+\int_{0}^{x}S_{\mathfrak{I}_{N}}(x,t)u(t)dt
\end{equation}
we obtain
\begin{equation}
\label{transmutedcosineandsinesol}c_{\mathfrak{I}_{N}}^{h}(\rho,x)=\mathbf{T}%
_{\mathfrak{I}_{N},h}^{C}\left[ \cos(\rho x)\right] , \quad s_{\mathfrak{I}%
_{N}}(\rho,x) = \mathbf{T}_{\mathfrak{I}_{N}}^{S}\left[ \frac{\sin(\rho
x)}{\rho}\right] .
\end{equation}

\end{remark}

\begin{remark}
\label{remarkwronskian}  According to Remark \ref{remarksinecosinefunda}, the
space of solutions of (\ref{Schrwithdelta}) has dimension 2, and given
$f,g\in\mathcal{D}_{2}\left( \mathbf{L}_{q,\mathfrak{I}_{N}}\right) $
solutions of (\ref{Schrwithdelta}), repeating the same procedure of Remark
\ref{remarksinecosinefunda}, $W[f,g]$ is constant in the whole segment
$[0,b]$. The solutions $f,g$ are a fundamental set of solutions iff
$W[f,g]\neq0$.
\end{remark}

\section{The SPPS method and the mapping property}

\subsection{Spectral parameter powers series}

As in the case of the regular Schr\"odinger equation
\cite{blancarte,sppsoriginal}, we obtain a representation for the solutions of
(\ref{Schrwithdelta}) as a power series in the spectral parameter (SPPS
series). Assume that there exists a solution $f\in\mathcal{D}_{2}\left(
\mathbf{L}_{q,\mathfrak{I}_{N}}\right) $ that does not vanish in the whole
segment $[0,b]$.

\begin{remark}
\label{remarknonhomeq}  Given $g\in L_{2}(0,b)$, a solution $u\in
\mathcal{D}_{2}\left( \mathbf{L}_{q,\mathfrak{I}_{N}}\right) $ of the
non-homogeneous Cauchy problem
\begin{equation}
\label{nonhomoCauchy}%
\begin{cases}
\mathbf{L}_{q,\mathfrak{I}_{N}}u(x)= g(x), \quad0<x<b\\
u(0)=u_{0}, \; u^{\prime}(0)=u_{1}%
\end{cases}
\end{equation}
can be obtained by solving the regular equation $\mathbf{L}_{q}u(x)=g(x)$ a.e.
$x\in(0,b)$ as follows. Consider the Polya factorization $\mathbf{L}%
_{q}u=-\frac{1}{f}Df^{2}D\frac{u}{f}$, where $D=\frac{d}{dx}$. A direct
computation shows that $u$ given by
\begin{equation}
\label{solutionnonhomogeneouseq}u(x)= -f(x)\int_{0}^{x}\frac{1}{f^{2}(t)}%
\int_{0}^{t} f(s)g(s)ds+\frac{u_{0}}{f(0)}f(x)+(f(0)u_{1}-f^{\prime}%
(0)u_{0})f(x)\int_{0}^{x}\frac{dt}{f^{2}(t)}%
\end{equation}
satisfies (\ref{nonhomoCauchy}) (actually, $f(x)\int_{0}^{x}\frac{1}{f^{2}%
(t)}dt$ is the second linearly independent solution of $\mathbf{L}_{q}u=0$
obtained from $f$ by Abel's formula). By Remark \ref{remarkidealdomain},
$u\in\mathcal{D}_{2}\left( \mathbf{L}_{q,\mathfrak{J}_{N}}\right) $ and by
Proposition \ref{propregular} and Remark \ref{remarkunicityofsol}, formula
(\ref{solutionnonhomogeneouseq}) provides the unique solution of
(\ref{nonhomoCauchy}). Actually, if we denote $\mathcal{I}u(x):= \int_{0}%
^{x}u(t)dt$ and define $\mathbf{R}_{\mathfrak{I}_{N}}^{f}:= -f\mathcal{I}%
f^{2}\mathcal{I}$, then $\mathbf{R}_{\mathfrak{I}_{N}}^{f}\in\mathcal{B}\left(
L_{2}(0,b)\right) $, $\mathbf{R}_{\mathfrak{I}_{N}}^{f}\left( L_{2}%
(0,b)\right) \subset\mathcal{D}_{2}\left( \mathbf{L}_{q,\mathfrak{I}_{N}%
}\right) $ and is a right-inverse for $\mathbf{L}_{q,\mathfrak{I}_{N}}$, i.e.,
$\mathbf{L}_{q,\mathfrak{I}_{N}}\mathbf{R}_{\mathfrak{I}_{N}}^{f} g=g$ for all
$g\in L_{2}(0,b)$.
\end{remark}

Following \cite{sppsoriginal} we define the following recursive integrals:
$\widetilde{X}^{(0)}\equiv X^{(0)} \equiv1$, and for $k\in\mathbb{N}$
%On each interval $(x_k,x_{k+1})$, $k=0,\cdots,N$, we have the Polya factorization $-y''(x)+q(x)y= -\frac{1}{f^2(x)}\frac{d}{dx}f^2(x)\frac{d}{dx}\frac{y(x)}{f(x)}$, for all $y\in H^2(x_k, x_{k+1})$. By Proposition \ref{propregular} we obtain the Polya factorization for $\mathbf{L}_{q, \mathfrak{J}_N}$ as follows
%\begin{equation}\label{polyafact}
%	(\mathbf{L}_{q,\mathfrak{J}_N}u, \varphi)_{C_0^{\infty}(0,b)}= -\left(\frac{1}{f^2}Df^2D\frac{u}{f},\varphi\right)_{C_0^{\infty}(0,b)}\quad \forall u\in \mathcal{D}_2\left(\mathbf{L}_{q,\mathfrak{J}_N}\right),
%\end{equation}
%where $D=\frac{d}{dx}$.
%

\begin{align}
\widetilde{X}^{(k)}(x)  &  :=  k\int_{0}^{x}\widetilde{X}^{(k-1)}(s)\left(
f^{2}(s)\right) ^{(-1)^{k-1}}ds,\\
X^{(k)}(x)  &  :=  k\int_{0}^{x}X^{(k-1)}(s)\left( f^{2}(s)\right) ^{(-1)^{k}%
}ds.
\end{align}
The functions $\{\varphi_{f}^{(k)}(x)\}_{k=0}^{\infty}$ defined by
\begin{equation}
\label{formalpowers}\varphi_{f}^{(k)}(x):=
\begin{cases}
f(x)\widetilde{X}^{(k)}(x),\quad\mbox{if } k \mbox{ even},\\
f(x)X^{(k)}(x),\quad\mbox{if } k \mbox{ odd}.
\end{cases}
\end{equation}
for $k\in\mathbb{N}_{0}$, are called the \textit{formal powers} associated to
$f$. Additionally, we introduce the following auxiliary formal powers
$\{\psi_{f}^{(k)}(x)\}_{k=0}^{\infty}$ given by
\begin{equation}
\label{auxiliaryformalpowers}\psi_{f}^{(k)}(x):=
\begin{cases}
\frac{\widetilde{X}^{(k)}(x)}{f(x)},\quad\mbox{if } k \mbox{ odd},\\
\frac{X^{(k)}(x)}{f(x)},\quad\mbox{if } k \mbox{ even}.
\end{cases}
\end{equation}

\begin{remark}
\label{remarkformalpowers}  For each $k \in\mathbb{N}_{0}$, $\varphi_{f}%
^{(k)}\in\mathcal{D}_{2}\left( \mathbf{L}_{q,\mathfrak{I}_{N}}\right) $.
Indeed, direct computations show that the following relations hold for all
$k\in\mathbb{N}_{0}$:
\begin{align}
D\varphi_{f}^{(k)} &  = \frac{f^{\prime}}{f}\varphi_{f}^{(k)}+k\psi
_{f}^{(k-1)}\label{formalpowersderivative}\\
D^{2}\varphi_{f}^{(k)}  &  = \frac{f^{\prime\prime}}{f}\varphi_{f}%
^{(k)}+k(k-1)\varphi_{f}^{(k-2)}\label{Lbasisproperty}%
\end{align}
Since $\varphi_{f}^{(k)}, \psi_{f}^{(k)}\in C[0,b]$, using the procedure from
Remark \ref{remarkidealdomain} and (\ref{formalpowersderivative}) we obtain
$\varphi_{f}^{(k)}\in\mathcal{D}_{2}\left( \mathbf{L}_{q,\mathfrak{I}_{N}%
}\right) $.
\end{remark}

\begin{theorem}
[SPPS method]\label{theoremspps}  Suppose that $f\in\mathcal{D}_{2}\left(
\mathbf{L}_{q,\mathfrak{I}_{N}}\right) $ is a solution of (\ref{Schrwithdelta}%
) that does not vanish in the whole segment $[0,b]$. Then the functions
\begin{equation}
\label{SPPSseries}u_{0}(\rho,x)= \sum_{k=0}^{\infty}\frac{(-1)^{k}\rho
^{2k}\varphi_{f}^{(2k)}(x)}{(2k)!}, \quad u_{1}(\rho,x)= \sum_{k=0}^{\infty
}\frac{(-1)^{k}\rho^{2k}\varphi_{f}^{(2k+1)}(x)}{(2k+1)!}%
\end{equation}
belong to $\mathcal{D}_{2}\left( \mathbf{L}_{q,\mathfrak{I}_{N}}\right) $, and
$\{u_{0}(\rho,x), u_{1}(\rho,x)\}$ is a fundamental set of solutions for
(\ref{Schrwithdelta}) satisfying the initial conditions
\begin{align}
u_{0}(\rho,0)=f(0),  &   u^{\prime}_{0}(\rho,0)=f^{\prime}%
(0),\label{initialspps0}\\
u_{1}(\rho,0)=0,  &   u^{\prime}_{1}(\rho,0)=\frac{1}{f(0)}%
,\label{initialspps1}%
\end{align}
The series in (\ref{SPPSseries}) converge absolutely and uniformly on
$x\in[0,b]$, the series of the derivatives converge in $L_{2}(0,b)$ and the
series of the second derivatives converge in $L_{2}(x_{j},x_{j+1})$, $j=0,
\cdots, N$. With respect to $\rho$ the series converge absolutely and
uniformly on any compact subset of the complex $\rho$-plane.
\end{theorem}

\begin{proof}
Since $f\in C[0,b]$, the following estimates for the recursive integrals
$\{\widetilde{X}^{(k)}(x)\}_{k=0}^{\infty}$ and $\{X^{(k)}(x)\}_{k=0}^{\infty
}$ are known:
\begin{equation}
\label{auxiliarestimatesspps}|\widetilde{X}^{(n)}(x)|\leqslant M_{1}^{n}
b^{n}, \; |X^{(n)}(x)|\leqslant M_{1}^{n} b^{n}\quad\mbox{for all } x\in[0,b],
\end{equation}
where $M_{1}=\|f^{2}\|_{C[0,b]}\cdot\left\| \frac{1}{f^{2}}\right\| _{C[0,b]}$
(see the proof of Theorem 1 of \cite{sppsoriginal}). Thus, by the Weierstrass
$M$-tests, the series in (\ref{SPPSseries}) converge absolutely and uniformly
on $x\in[0,b]$, and for $\rho$ on any compact subset of the complex $\rho
$-plane. We prove that $u_{0}(\rho,x)\in\mathcal{D}_{2}\left( \mathbf{L}%
_{q,\mathfrak{I}_{N}}\right) $ and is a solution of (\ref{Schrwithdelta}) (the
proof for $u_{1}(\rho,x)$ is analogous). By Remark \ref{remarkformalpowers},
the series of the derivatives of $u_{0}(\rho,x)$ is given by $\frac{f^{\prime
}}{f}\sum_{k=0}^{\infty}\frac{(-1)^{k}\rho^{2k}\varphi_{f}^{(2k)}}{(2k)!}%
+\sum_{k=1}^{\infty}\frac{(-1)^{k}\rho^{2k}\psi_{f}^{(2k-1)}}{(2k-1)!}$. By
(\ref{auxiliarestimatesspps}), the series involving the formal powers
$\varphi_{f}^{(k)}$ and $\psi_{f}^{(k)}$ converge absolutely and uniformly on
$x\in[0,b]$. Hence, $\sum_{k=0}^{\infty}\frac{(-1)^{k}\rho^{k}D\varphi
_{f}^{(2k)}(x)}{(2k)!}$ converges in $L_{2}(0,b)$. Due to \cite[Prop.
3]{blancarte}, $u_{0}(\rho,\cdot)\in AC[0,b]$ and $u_{0}^{\prime}%
(\rho,x)=\frac{f^{\prime}(x)}{f(x)}\sum_{k=0}^{\infty}\frac{(-1)^{k}\rho
^{2k}\varphi_{f}^{(2k)}}{(2k)!}+\sum_{k=1}^{\infty}\frac{(-1)^{k}\rho^{2k}%
\psi_{f}^{(2k-1)}}{(2k-1)!}$ in $L_{2}(0,b)$. Since the series involving the
formal powers defines continuous functions, then $u_{0}(\rho,x)$ satisfies the
jump condition (\ref{jumpderivative}). Applying the same reasoning it is shown
that $u_{0}^{\prime\prime}(\rho,x)=\sum_{k=0}^{\infty}\frac{(-1)^{k}\rho
^{2k}D^{2}\varphi_{f}^{(2k)}}{(2k)!}$, the series converges in $L_{2}%
(x_{j},x_{j+1})$ and $u_{0}(\rho,\cdot)|_{(x_{j},x_{j+1})}\in H^{2}%
(x_{j},x_{j+1})$, $j=0, \dots, N$.

Since $\widetilde{X}^{(n)}(0)=0$ for $n\geqslant1$, we have
(\ref{initialspps0}). Finally, by (\ref{Lbasisproperty})
\begin{align*}
\mathbf{L}_{q}u_{0}(\rho,x)  &  = \sum_{k=0}^{\infty}\frac{(-1)^{k}\rho
^{2k}\mathbf{L}_{q}\varphi_{f}^{(2k)}(x)}{(2k)!}=\sum_{k=2}^{\infty}%
\frac{(-1)^{k+1}\rho^{2k}\varphi_{f}^{(2k-2)}(x)}{(2k-2)!}\\
& = \rho^{2}\sum_{k=0}^{\infty}\frac{(-1)^{k}\rho^{2k}\varphi_{f}^{(2k)}%
(x)}{(2k)!}=\rho^{2}u(\rho,x),
\end{align*}
this for a.e. $x\in(x_{j},x_{j+1})$, $j=0, \dots, N$.

Using (\ref{initialspps0}) and (\ref{initialspps1}) we obtain $W[u_{0}%
(\rho,x),u_{1}(\rho,x)](0)=1$. Since the Wronskian is constant (Remark
\ref{remarkwronskian}), $\{u_{0}(\rho,x), u_{1}(\rho,x)\}$ is a fundamental
set of solutions.
\end{proof}

\subsection{Existence and construction of the non-vanishing solution}

The existence of a non-vanishing solution is well known for the case of a
regular Schr\"odinger equation with continuous potential (see \cite[Remark
5]{sppsoriginal} and \cite[Cor. 2.3]{camporesi}). The following proof adapts
the one presented in \cite[Prop. 2.9]{nelsonspps} for the Dirac system.

\begin{proposition}
[Existence of non-vanishing solutions]\label{propnonvanishingsol}  Let
$\{u,v\}\in\mathcal{D}_{2}\left( \mathbf{L}_{q,\mathfrak{I}_{N}}\right) $ be a
fundamental set of solutions for (\ref{Schrwithdelta}). Then there exist
constants $c_{1},c_{2}\in\mathbb{C}$ such that the solution $f=c_{1}u+c_{2}v$
does not vanish in the whole segment $[0,b]$.
\end{proposition}

\begin{proof}
Let $\{u,v\}\in\mathcal{D}_{2}\left( \mathbf{L}_{q,\mathfrak{I}_{N}}\right) $
be a fundamental set of solutions for (\ref{Schrwithdelta}). Then $u$ and $v$
cannot have common zeros in $[0,b]$. Indeed, if $u(\xi)=v(\xi)=0$ for some
$\xi\in[0,b]$, then $W[u,v](\xi+)=u(\xi)v^{\prime}(\xi+)-v(\xi)u^{\prime}%
(\xi+)=0$. Since $W[u,v]$ is constant in $[0,b]$, this contradicts that
$\{u,v\}$ is a fundamental system.

This implies that in each interval $[x_{j},x_{j+1}]$, $j=0, \cdots, N$, the
map $F_{j}: [x_{j}, x_{j+1}]\rightarrow\mathbb{CP}^{1}$, $F_{j}(x):=\left[
u|_{[x_{j}, x_{j+1}]}(x) :v|_{[x_{j}, x_{j+1}]}(x)\right] $ (where
$\mathbb{CP}^{1}$ is the complex projective line, i.e., the quotient of
$\mathbb{C}^{2}\setminus\{(0,0)\}$ under the action of $\mathbb{C}^{*}$, and
$[a:b]$ denotes the equivalent class of the pair $(a,b)$) is well defined and
differentiable. In \cite[Prop. 2.2]{camporesi} it was established that a
differentiable function $f: I\rightarrow\mathbb{CP}^{1}$, where $I\subset
\mathbb{R}$ is an interval, is never surjective, using that Sard's theorem
implies that $f(I)$ has measure zero.

Suppose that $(\alpha,\beta)\in\mathbb{C}^{2}\setminus\{(0,0)\}$ is such that
$\alpha u(\xi)-\beta v(\xi)=0$ for some $\xi\in[0,b]$. Hence  $%
\begin{vmatrix}
u(\xi) & \beta\\
v(\xi) & \alpha
\end{vmatrix}
=0$, that is, $(u(\xi), v(\xi))$ and $(\alpha,\beta)$ are proportional. Since
$\xi\in[x_{j},x_{j+1}]$ for some $j\in\{0,\cdots, N\}$, hence $[\alpha
:-\beta]\in F_{j}\left(  [x_{j}, x_{j+1}]\right) $.

Thus, the set $C:= \left\{  [\alpha:\beta]\in\mathbb{CP}^{1}\, |\,\exists
\xi\in[0,b] \,:\, \alpha u(\xi)+\beta v(\xi)=0\right\} $ is contained in
\newline$\cup_{j=0}^{N}F_{j}\left(  [x_{j}, x_{j+1}]\right) $, and then $C$
has measure zero. Hence we can obtain a pair of constants $(c_{1},c_{2})
\in\mathbb{C}^{2}\setminus\{ (0,0)\}$ with $[c_{1}:-c_{2}]\in\mathbb{CP}%
^{1}\setminus C$ and $f=c_{1}u+c_{2}v$ does not vanish in the whole segment
$[0,b]$.
\end{proof}

\begin{remark}
If $q$ is real valued and $\alpha_{1}, \cdots, \alpha_{N}\in\mathbb{R}%
\setminus\{0\}$, taking a real-valued fundamental system of solutions for the
regular equation $\mathbf{L}_{q} y=0$ and using formula
(\ref{generalsolcauchy}), we can obtain a real-valued fundamental set of
solutions $\{u,v\}$ for $\mathbf{L}_{q,\mathfrak{I}_{N}}y=0$. In the proof of
Proposition \ref{propnonvanishingsol} we obtain that $u$ and $v$ have no
common zeros. Hence $f=u+iv$ is a non vanishing solution.

For the complex case, we can choose randomly a pair of constants $(c_{1}%
,c_{2})\in\mathbb{C}^{2}\setminus\{(0,0)\}$ and verify if the linear
combination $c_{1}u+c_{2}v$ has no zero. If there is a zero, we repeat the
process until we find the non vanishing solution. Since the set $C$ (from the
proof of Proposition \ref{propnonvanishingsol}) has measure zero, is almost
sure to find the coefficients $c_{1},c_{2}$ in the first few tries.
\end{remark}

By Proposition \ref{propnonvanishingsol}, there exists a pair of constants
$(c_{1},c_{2})\in\mathbb{C}^{2}\setminus\{(0,0)\}$ such that
\begin{align}
y_{0}(x) &  = c_{1}+c_{2}x+\sum_{k=1}^{N}\alpha_{k}(c_{1}+c_{2}x_{k}%
)H(x-x_{k})(x-x_{k})\nonumber\\
& \quad+\sum_{J\in\mathcal{J}_{N}}\alpha_{J}(c_{1}+c_{2}x_{j_{1}%
})H(x-x_{j_{|J|}})\left( \prod_{l=1}^{|J|-1}(x_{j_{l+1}}-x_{j_{1}})\right)
(x-x_{j_{|J|}})\label{nonvanishingsolqzero}%
\end{align}
is a non-vanishing solution of (\ref{Schrwithdelta}) for $\rho=0$ (if
$\alpha_{1}, \dots,\alpha_{k} \in(0,\infty)$, it is enough with take $c_{1}%
=1$, $c_{2}=0$). Below we give a procedure based on the SPPS method
(\cite{blancarte,sppsoriginal}) to obtain the non-vanishing solution $f$ from
$y_{0}$.

\begin{theorem}
Define the recursive integrals $\{Y^{(k)}\}_{k=0}^{\infty}$ and $\{\tilde
{Y}^{(k)}\}_{k=0}^{\infty}$ as follows: $Y^{(0)}\equiv\tilde{Y}^{(0)}\equiv1$,
and for $k\geqslant1$
\begin{align}
Y^{(k)}(x)  &  =
\begin{cases}
\int_{0}^{x} Y^{(k)}(s)q(s)y_{0}^{2}(s)ds, & \mbox{ if  } k \mbox{ is even},\\
\int_{0}^{x} \frac{Y^{(k)}(s)}{y_{0}^{2}(s)}ds, & \mbox{ if  } k
\mbox{ is odd},
\end{cases}
\\
\tilde{Y}^{(k)}(x)  &  =
\begin{cases}
\int_{0}^{x} \tilde{Y}^{(k)}(s)q(s)y_{0}^{2}(s)ds, & \mbox{ if  } k
\mbox{ is odd},\\
\int_{0}^{x} \frac{\tilde{Y}^{(k)}(s)}{y_{0}^{2}(s)}ds, & \mbox{ if  } k
\mbox{ is even}.
\end{cases}
\end{align}
Define
\begin{equation}
\label{sppsforparticularsol}f_{0}(x)= y_{0}(x)\sum_{k=0}^{\infty}\tilde
{Y}^{(2k)}(x), \qquad f_{1}(x) = y_{0}(x)\sum_{k=0}^{\infty}Y^{(2k+1)}(x).
\end{equation}
Then $\{f_{0},f_{1}\}\subset\mathcal{D}_{2}\left( \mathbf{L}_{q,
\mathfrak{I}_{N}}\right) $ is a fundamental set of solution for $\mathbf{L}%
_{q,\mathfrak{I}_{N}}u=0$ satisfying the initial conditions $f_{0}(0)=c_{1}$,
$f^{\prime}_{0}(0)=c_{2}$, $f_{1}(0)=0$, $f^{\prime}_{1}(0)=1$. Both series
converge uniformly and absolutely on $x\in[0,b]$. The series of the
derivatives converge in $L_{2}(0,b)$, and on each interval $[x_{j},x_{j+1}]$,
$j=0, \dots, N$, the series of the second derivatives converge in $L_{2}%
(x_{j}, x_{j+1})$. Hence there exist constants $C_{1},C_{2}\in\mathbb{C}$ such
that $f=C_{1}f_{0}+C_{2}f_{1}$ is a non-vanishing solution of $\mathbf{L}_{q,
\mathfrak{I}_{N}}u=0$ in $[0,b]$.
\end{theorem}

\begin{proof}
Using the estimates
\[
|\tilde{Y}^{(2k-j)}(x)|\leqslant\frac{M_{1}^{(n-j)}M_{2}^{n}}{(n-j)!n!},
\quad|Y^{(2k-j)}(x)|\leqslant\frac{M_{1}^{n}M_{2}^{(n-j)}}{n!(n-j)!}, \quad
x\in[0,b], \; j=0,1,\; k \in\mathbb{N},
\]
where $M_{1}= \left\| \frac{1}{y_{0}^{2}}\right\| _{L_{1}(0,b)}$ and $M_{2}=
\|qy_{0}^{2}\|_{L_{1}(0,b)}$, from \cite[Prop. 5]{blancarte}, the series in
(\ref{sppsforparticularsol}) converge absolutely and uniformly on $[0,b]$. The
proof of the convergence of the derivatives and that $\{f_{0},f_{1}%
\}\in\mathcal{D}_{2}\left( \mathbf{L}_{q, \mathfrak{I}_{N}}\right) $ is a
fundamental set of solutions is analogous to that of Theorem \ref{theoremspps}
(see also \cite[Th. 1]{sppsoriginal}) and \cite[Th. 7]{blancarte} for the
proof in the regular case).
\end{proof}

\subsection{The mapping property}

Take a non vanishing solution $f\in\mathcal{D}_{2}\left( \mathbf{L}%
_{q,\mathfrak{I}_{N}}\right) $ normalized at zero, i.e., $f(0)=1$, and set
$h=f^{\prime}(0)$. Then the corresponding transmutation operator and kernel
$\mathbf{T}^{h}_{\mathfrak{I}_{N}}$ and $K_{\mathfrak{I}_{N}}^{h}(x,t)$ will
be denoted by $\mathbf{T}^{f}_{\mathfrak{I}_{N}}$ and $K_{\mathfrak{I}_{N}%
}^{f}(x,t)$ and called the \textit{canonical} transmutation operator and
kernel associated to $f$, respectively (same notations are used for the cosine
and sine transmutations).

\begin{theorem}
\label{theoretransprop}  The canonical transmutation operator $\mathbf{T}%
^{f}_{\mathfrak{I}_{N}}$ satisfies the following relations
\begin{equation}
\label{transmproperty}\mathbf{T}_{\mathfrak{I}_{N}}^{f}\left[ x^{k}\right] =
\varphi_{f}^{(k)}(x) \qquad\forall k\in\mathbb{N}_{0}.
\end{equation}
The canonical cosine and sine transmutation operators satisfy the relations
\begin{align}
\mathbf{T}_{\mathfrak{I}_{N},f}^{C}\left[ x^{2k}\right]   &  =  \varphi
_{f}^{(2k)}(x) \qquad\forall k\in\mathbb{N}_{0}.\label{transmpropertycos}\\
\mathbf{T}_{\mathfrak{I}_{N}}^{S}\left[ x^{2k+1}\right]   &  =  \varphi
_{f}^{(2k+1)}(x) \qquad\forall k\in\mathbb{N}_{0}.\label{transmpropertysin}%
\end{align}

\end{theorem}

\begin{proof}
Consider the solution $e_{\mathfrak{I}_{N}}^{h}(\rho,x)$ with $h=f^{\prime
}(0)$. By the conditions (\ref{initialspps0}) and (\ref{initialspps1}),
solution $e_{\mathfrak{I}_{N}}^{h}(\rho,x)$ can be written in the form
\begin{align}
e_{\mathfrak{I}_{N}}^{h}(\rho,x)  &  = u_{0}(\rho,x)+i\rho u_{1}%
(\rho,x)\nonumber\\
& = \sum_{k=0}^{\infty}\frac{(-1)^{k}\rho^{2k}\varphi_{f}^{(2k)}(x)}%
{(2k)!}+\sum_{k=0}^{\infty}\frac{i(-1)^{k}\rho^{2k+1}\varphi_{f}^{(2k+1)}%
(x)}{(2k+1)!}\nonumber\\
&  = \sum_{k=0}^{\infty}\frac{(i\rho)^{2k}\varphi_{f}^{(2k)}(x)}{(2k)!}%
+\sum_{k=0}^{\infty}\frac{(i\rho)^{2k+1}\varphi_{f}^{(2k+1)}(x)}%
{(2k+1)!}\nonumber\\
& = \sum_{k=0}^{\infty}\frac{(i\rho)^{k}\varphi_{f}^{(k)}(x)}{k!}%
\label{auxseriesthmap}%
\end{align}
(The rearrangement of the series is due to absolute and uniform convergence,
Theorem \ref{theoremspps}). On the other hand
\[
e_{\mathfrak{I}_{N}}^{h}(\rho,x) = \mathbf{T}_{\mathfrak{I}_{N}}^{f}\left[
e^{i\rho x}\right]  = \mathbf{T}_{\mathfrak{I}_{N}}^{f}\left[  \sum
_{k=0}^{\infty}\frac{(i\rho)^{k}x^{k}}{k!} \right]
\]
Note that $\displaystyle \int_{-x}^{x}K_{\mathfrak{I}_{N}}^{f}(x,t)\left(
\sum_{k=0}^{\infty}\frac{(i\rho)^{k}t^{k}}{k!} \right) dt= \sum_{k=0}^{\infty
}\frac{(i\rho)^{k}}{k!}\int_{-x}^{x}K_{\mathfrak{I}_{N}}^{f}(x,t)t^{k}dt$, due
to the uniform convergence of the exponential series in the variable $t
\in[-x,x]$. Thus,
\begin{equation}
\label{auxiliarseriesthmap}e_{\mathfrak{I}_{N}}^{h}(\rho,x) = \sum
_{k=0}^{\infty}\frac{(i\rho)^{k}\mathbf{T}_{\mathfrak{I}_{N}}^{f}\left[
x^{k}\right] }{k!}.
\end{equation}
Comparing (\ref{auxiliarseriesthmap}) and (\ref{auxseriesthmap}) as Taylor
series in the complex variable $\rho$ we obtain (\ref{transmproperty}).
Relations (\ref{transmpropertycos}) and (\ref{transmpropertysin}) follows from
(\ref{transmproperty}), (\ref{cosineintegralkernelgral}),
(\ref{sineintegralkernelgral}) and the fact that $G_{\mathfrak{I}_{N}}%
^{f}(x,t)$ and $S_{\mathfrak{I}_{N}}(x,t)$ are even and odd in the variable
$t$, respectively.
\end{proof}

\begin{remark}
\label{remarkformalpowersasymp}  The formal powers $\{\varphi_{f}%
^{(k)}(x)\}_{k=0}^{\infty}$ satisfy the asymptotic relation\newline%
$\varphi_{f}^{(k)}(x)=x^{k}(1+o(1))$, $x\rightarrow0^{+}$, $\forall
k\in\mathbb{N}$.

Indeed, by Theorem \ref{theoretransprop} and the Cauchy-Bunyakovsky-Schwarz
inequality we have
\begin{align*}
|\varphi_{f}^{(k)}(x)-x^{k}|  &  = \left| \int_{-x}^{x}K_{\mathfrak{I}_{N}%
}^{f}(x,t)t^{k}dt\right|  \leqslant\left( \int_{-x}^{x}\left| K_{\mathfrak{I}%
_{N}}^{f}(x,t)\right| ^{2}dt\right) ^{\frac{1}{2}}\left( \int_{-x}^{x}%
|t|^{2k}dt\right) ^{\frac{1}{2}}\\
& \leqslant\sqrt{2b}\left\| K_{\mathfrak{I}_{N}^{f}}\right\| _{L_{\infty
}(\Omega)}\sqrt{\frac{2}{2k+1}}x^{k+\frac{1}{2}}%
\end{align*}
(because $K_{\mathfrak{I}_{N}}^{f}\in L_{\infty}(\Omega)$ by Theorem
\ref{thoremtransmoperator}). Hence
\[
\left| \frac{\varphi_{f}^{(k)}(x)}{x^{k}}-1\right| \leqslant\sqrt{2b}\left\|
K_{\mathfrak{I}_{N}^{f}}\right\| _{L_{\infty}(\Omega)}\sqrt{\frac{2}{2k+}%
}x^{\frac{1}{2}} \rightarrow0, \qquad x\rightarrow0^{+}.
\]

\end{remark}

\begin{remark}
\label{remarktransmoperatorlbais} Denote $\mathcal{P}(\mathbb{R}%
)=\mbox{Span}\{x^{k}\}_{k=0}^{\infty}$. According to Remark
\ref{remarkformalpowers} and Proposition \ref{propregular} we have that
$\mathbf{T}_{\mathfrak{I}_{N}}^{f}\left( \mathcal{P}(\mathbb{R})\right)
=\mbox{Span}\left\{ \varphi_{f}^{(k)}(x)\right\} _{k=0}^{\infty}$, and by
(\ref{Lbasisproperty}) we have the relation%

\begin{equation}
\label{transmrelationdense}\mathbf{L}_{q,\mathfrak{I}_{N}}\mathbf{T}%
_{\mathfrak{I}_{N}}^{f} p = -\mathbf{T}_{\mathfrak{I}_{N}}^{f}D^{2}p
\qquad\forall p \in\mathcal{P}(\mathbb{R}).
\end{equation}

According to \cite{hugo2}, $\mathbf{T}_{q,\mathfrak{I}_{N}}^{f}$ is a
transmutation operator for the pair $\mathbf{L}_{q,\mathfrak{I}_{N}}$,
$-D^{2}$ in the subspace $\mathcal{P}(\mathbb{R})$, and $\{\varphi_{f}%
^{(k)}(x)\}_{k=0}^{\infty}$ is an $\mathbf{L}_{q,\mathfrak{I}_{N}}$-basis.
Since $\varphi_{f}^{(K)}(0)=D\varphi_{f}^{(k)}(0)=0$ for $k\geqslant2$,
$\{\varphi_{f}^{(k)}(x)\}_{k=0}^{\infty}$ is called a \textbf{standard }
$\mathbf{L}_{q,\mathfrak{J}_{N}}$-basis, and $\mathbf{T}_{\mathfrak{I}_{N}%
}^{f}$ a standard transmutation operator. By Remark \ref{remarknonhomeq} we
can recover $\varphi_{f}^{(k)}$ for $k\geqslant2$ from $\varphi_{f}^{(0)}$ and
$\varphi_{f}^{(0)}$ by the formula
\begin{equation}
\label{recorverformalpowers}\varphi_{f}^{(k)} (x) =-k(k-1)\mathbf{R}%
_{\mathfrak{I}_{N}}^{f}\varphi_{f}^{(k)}(x) =k(k-1)f(x)\int_{0}^{x}\frac
{1}{f^{2}(t)}\int_{0}^{t}f(s)\varphi_{f}^{(k-2)}(s)ds
\end{equation}
(compare this formula with \cite[Formula (8), Remark 9]{hugo2}).
\end{remark}

%\begin{remark}
%	Consider the kernel $G(x,t)= \varphi_f^{(0)}(x)\varphi_f^{(1)}(t)-\varphi_f^{(0)}(t)\varphi_f^{(1)}(x)$. Given $g\in L_2(0,b)$, define
%	\begin{equation}\label{greenfunctionformula}
%		u(x)= \int_0^xG(x,t)g(t)dt
%	\end{equation}
%Note that $u\in AC[0,b]$ with $	u'(x)= \varphi_f^{(0)}(x)\int_0^x\varphi_f^{(1)}(t)g(t)dt-\varphi_f^{(1)}(x)\int_0^x\varphi_f^{(0)}(t)g(t)dt$. Since the integrals define continuous functions, by Remark \ref{remarkidealdomain} $u$ satisfies condition (\ref{jumpderivative}). For the second derivative we have
%\begin{align*}
%	u''(x) & = D\varphi_f^{(0)}(x)\varphi_f^{(1)}(x)g(x)+D^2\varphi_f^{(0)}(x)\int_0^x\varphi_f^{(1)}(t)g(t)dt-D\varphi_f^{(1)}(x)\varphi_f^{(0)}(x)g(x)\\
%	& \quad -D^2\varphi_f^{(1)}(x)\int_0^x\varphi_f^{(0)}(t)g(t)dt \\
%	&  = -W\left[\varphi_f^{(0)}, \varphi_f^{(1)}\right](x)g(x)+D^2\varphi_f^{(0)}(x)\int_0^x\varphi_f^{(1)}(t)g(t)dt-D^2\varphi_f^{(1)}(x)\int_0^x\varphi_f^{(0)}(t)g(t)dt
%\end{align*}
%and then $u\in \mathcal{D}_2\left(\mathbf{L}_{q, \mathfrak{J}_N}\right)$. Since $W\left[\varphi_f^{(0)}, \varphi_f^{(1)}\right]\equiv 1$, we get $\mathbf{L}_q u= g$. By proposition \ref{propregular}, $u$ is a solution of $\mathbf{L}_{q,\mathfrak{J}_N}u=g$.

%Actually, since $\{\varphi_f^{(k)}(x)\}_{k=0}^{\infty}$ is a standard $\mathbf{L}_{q,\mathfrak{J}_N}$-basis, the following formula is valid (see \cite[Remark 9]{hugo2})
%\begin{equation}
%	\varphi_f^{(k)}(x)= k(k-1)\int_0^xG(x,t)\varphi_f^{(k-2)}(t)dt\quad \forall k\geqslant 2.
%\end{equation}
%\end{remark}

The following result adapts Theorem 10 from \cite{hugo2}, proved for the case
of an $L_{1}$-regular potential.

\begin{theorem}
The operator $\mathbf{T}_{\mathfrak{I}_{N}}^{f}$ is a transmutation operator
for the pair $\mathbf{L}_{q, \mathfrak{I}_{N}}$, $-D^{2}$ in $H^{2}(-b,b)$,
that is, $\mathbf{T}_{\mathfrak{I}_{N}}^{f}\left( H^{2}(-b,b)\right)
\subset\mathcal{D}_{2}\left( \mathbf{L}_{q,\mathfrak{I}_{N}}\right) $ and
\begin{equation}
\label{transmpropertygeneral}\mathbf{L}_{q, \mathfrak{I}_{N}} \mathbf{T}%
_{\mathfrak{I}_{N}}u= -\mathbf{T}_{\mathfrak{I}_{N}}D^{2}u \qquad\forall u\in
H^{2}(-b,b)
\end{equation}

\end{theorem}

\begin{proof}
We show that
\begin{equation}
\label{auxiliareqtransmpropgen}\mathbf{T}_{\mathfrak{I}_{N}}u(x) =
u(0)\varphi_{f}^{(0)}(x)+u^{\prime}(0)\varphi_{f}^{(1)}(x)-\mathbf{R}%
_{\mathfrak{I}_{N}}^{f}\mathbf{T}_{\mathfrak{I}_{N}}^{f}u^{\prime\prime
2}(-b,b).
\end{equation}
Let us first see that (\ref{auxiliareqtransmpropgen}) is valid for
$p\in\mathcal{P}(\mathbb{R})$. Indeed, set $p(x)=\sum_{k=0}^{M}c_{k} x^{k}$.
By the linearity of $\mathbf{T}_{\mathfrak{I}_{N}}^{f}$, Theorem
\ref{theoretransprop} and (\ref{recorverformalpowers}) we have
\begin{align*}
\mathbf{T}_{\mathfrak{I}_{N}}^{f}p(x)  &  = c_{0}\varphi_{f}^{(0)}%
+c_{1}\varphi_{f}^{(1)}(x)+\sum_{k=2}^{M}c_{k}\varphi_{f}^{(k)}(x)\\
&  = p(0)\varphi_{f}^{(0)}+p^{\prime}(0)\varphi_{f}^{(1)}(x)-\sum_{k=2}%
^{M}c_{k}k(k-1)\mathbf{R}_{\mathfrak{I}_{N}}^{f}\varphi_{f}^{(k-2)}(x)\\
&  = p(0)\varphi_{f}^{(0)}+p^{\prime}(0)\varphi_{f}^{(1)}(x)-\sum_{k=2}%
^{M}c_{k}k(k-1)\mathbf{R}_{\mathfrak{I}_{N}}^{f}\mathbf{T}_{\mathfrak{I}_{N}%
}^{f} \left[ x^{k-2}\right] \\
&  = p(0)\varphi_{f}^{(0)}+p^{\prime}(0)\varphi_{f}^{(1)}(x)-\mathbf{R}%
_{\mathfrak{I}_{N}}^{f}\mathbf{T}_{\mathfrak{I}_{N}}^{f}p^{\prime\prime}(x)
\end{align*}
This establishes (\ref{auxiliareqtransmpropgen}) for $p\in\mathcal{P}%
(\mathbb{R})$. Take $u\in H^{2}(-b,b)$ arbitrary. There exists a sequence
$\{p_{n}\}\subset\mathcal{P}(\mathbb{R})$ such that $p_{n}^{(j)}%
\overset{[-b,b]}{\rightrightarrows} u^{(j)}$, $j=0,1$, and $p^{\prime\prime
}_{n}\rightarrow u$ in $L_{2}(-b,b)$, when $n\rightarrow\infty$ (see
\cite[Prop. 4]{hugo2}). Since $\mathbf{R}_{\mathfrak{I}_{N}}^{f}%
\mathbf{T}_{\mathfrak{I}_{N}}^{f}\in\mathcal{B}\left( L_{2}(-b,b),
L_{2}(0,b)\right) $ we have
\begin{align*}
\mathbf{T}_{\mathfrak{I}_{N}}^{f}u(x)  &  = \lim_{n\rightarrow\infty}
\mathbf{T}_{\mathfrak{I}_{N}}^{f}p_{n}(x) = \lim_{n\rightarrow\infty}\left[
p_{n}(0)\varphi_{f}^{(0)}+p^{\prime}_{n}(0)\varphi_{f}^{(1)}(x)-\mathbf{R}%
_{\mathfrak{I}_{N}}^{f}\mathbf{T}_{\mathfrak{I}_{N}}^{f}p^{\prime\prime}%
_{n}(x) \right] \\
&  = u(0)\varphi_{f}^{(0)}(x)+u^{\prime}(0)\varphi_{f}^{(1)}(x)-\mathbf{R}%
_{\mathfrak{I}_{N}}^{f}\mathbf{T}_{\mathfrak{I}_{N}}^{f}u^{\prime\prime}(x)
\end{align*}
and we obtain (\ref{auxiliareqtransmpropgen}). Hence, by Remark
\ref{remarknonhomeq}, $\mathbf{T}_{\mathfrak{I}_{N}}^{f}\left( H^{2}%
(-b,b)\right) \subset\mathcal{D}_{2}\left( \mathbf{L}_{q,\mathfrak{I}_{N}%
}\right) $, and since $\mathbf{L}_{q,\mathfrak{I}_{N}}\varphi_{f}^{(k)}=0$ for
$k=0,1$, applying $\mathbf{L}_{q,\mathfrak{I}_{N}}$ in both sides of
(\ref{auxiliareqtransmpropgen}) we have (\ref{transmpropertygeneral}).
\end{proof}

\section{Fourier-Legendre and Neumann series of Bessel functions expansions}

\subsection{Fourier-Legendre series expansion of the transmutation kernel}

Fix $x\in(0,b]$. Theorem \ref{thoremtransmoperator} establishes that
$K_{\mathfrak{I}_{N}}^{h}(x,\cdot)\in L_{2}(-x,x)$, then $K_{\mathfrak{I}_{N}%
}^{h}(x,t)$ admits a Fourier series in terms of an orthogonal basis of
$L_{2}(-x,x)$. Following \cite{neumann}, we choose the orthogonal basis of
$L_{2}(-1,1)$ given by the Legendre polynomials $\{P_{n}(z)\}_{n=0}^{\infty}$.
Thus,
\begin{equation}
\label{FourierLegendreSeries1}K_{\mathfrak{I}_{N}}^{h}(x,t)= \sum
_{n=0}^{\infty}\frac{a_{n}(x)}{x}P_{n}\left( \frac{t}{x}\right)
\end{equation}
where
\begin{equation}
\label{FourierLegendreCoeff1}a_{n}(x)=\left( n+\frac{1}{2}\right) \int
_{-x}^{x}K_{\mathfrak{I}_{N}}^{h}(x,t)P_{n}\left( \frac{t}{x}\right)
dt\qquad\forall n\in\mathbb{N}_{0}.
\end{equation}
The series (\ref{FourierLegendreSeries1}) converges with respect to $t$ in the
norm of $L_{2}(-x,x)$. Formula (\ref{FourierLegendreCoeff1}) is obtained
multiplying (\ref{FourierLegendreSeries1}) by $P_{n}\left( \frac{t}{x}\right)
$, using the general Parseval's identity \cite[pp. 16]{akhiezer} and taking
into account that $\|P_{n}\|_{L_{2}(-1,1)}^{2}=\frac{2}{2n+1}$, $n\in
\mathbb{N}_{0}$.

\begin{example}
Consider the kernel $K_{\mathfrak{I}_{1}}^{0}(x,t)=\frac{\alpha_{1}}%
{2}H(x-x_{1})\chi_{[2x_{1}-x,x]}$ from Example \ref{beginexample1}. In this
case, the Fourier-Legendre coefficients has the form
\[
a_{n}(x) = \frac{\alpha_{1}}{2}\left( n+\frac{1}{2}\right)  H(x-x_{1}%
)\int_{2x_{1}-x}^{x}P_{n}(t)dt=\frac{\alpha_{1}}{2}\left( n+\frac{1}%
{2}\right)  xH(x-x_{1})\int_{2\frac{x_{1}}{x}-1}^{1}P_{n}(t)dt.
\]
From this we obtain $a_{0}(x)=\frac{\alpha_{1}}{2}H(x-x_{1})(x-x_{1})$. Using
formula $P_{n}(t)=\frac{1}{2n+1}\frac{d}{dt}\left( P_{n+1}(t)-P_{n-1}%
(t)\right) $ for $n\in\mathbb{N}$, and that $P_{n}(1)=0$ for all
$n\in\mathbb{N}$, we have
\[
a_{n}(x)= \frac{\alpha_{1}}{4}xH(x-x_{1})\left[ P_{n-1}\left( \frac{2x_{1}}%
{x}-1\right) -P_{n+1}\left( \frac{2x_{1}}{x}-1\right) \right]
\]

\end{example}

\begin{remark}
From (\ref{FourierLegendreCoeff1}) we obtain that the first coefficient
$a_{0}(x)$ is given by
\begin{align*}
a_{0}(x)  &  = \frac{1}{2}\int_{-x}^{x}K_{\mathfrak{I}_{N}}^{h}(x,t)P_{0}%
\left( \frac{t}{x}\right) dt =\frac{1}{2}\int_{-x}^{x}K_{\mathfrak{I}_{N}}%
^{h}(x,t)dt\\
& = \frac{1}{2}\mathbf{T}_{\mathfrak{I}_{N}}^{h}[1]-\frac{1}{2} = \frac{1}%
{2}(e_{\mathfrak{I}_{N}}^{h}(0,x)-1).
\end{align*}
Thus, we obtain the relations
\begin{equation}
\label{relationfirstcoeff1}a_{0}(x)= \frac{1}{2}(e_{\mathfrak{I}_{N}}%
^{h}(0,x)-1), \qquad e_{\mathfrak{I}_{N}}^{h}(0,x)= 2a_{n}(x)+1.
\end{equation}

\end{remark}

For the kernels $G_{\mathfrak{I}_{N}}^{h}(x,t)$ and $S_{\mathfrak{I}_{N}%
}(x,t)$ we obtain the series representations in terms of the even and odd
Legendre polynomials, respectively,%

\begin{align}
G_{\mathfrak{I}_{N}}^{h}(x,t)  &  =  \sum_{n=0}^{\infty}\frac{g_{n}(x)}%
{x}P_{2n}\left( \frac{t}{x}\right) ,\label{Fourierexpansioncosine}\\
S_{\mathfrak{I}_{N}}(x,t)  &  =  \sum_{n=0}^{\infty}\frac{s_{n}(x)}{x}%
P_{2n+1}\left( \frac{t}{x}\right) ,\label{Fourierexpansionsine}%
\end{align}
where the coefficients are given by
\begin{align}
g_{n}(x)  &  =  2a_{2n}(x)= (4n+1)\int_{0}^{x}G_{\mathfrak{I}_{N}}%
^{h}(x,t)P_{2n}\left( \frac{t}{x}\right)
dt,\label{Fourierexpansioncosinecoeff}\\
s_{n}(x)  &  =  2a_{2n+1}(4n+3)\int_{0}^{x}S_{\mathfrak{I}_{N}}(x,t)P_{2n+1}%
\left( \frac{t}{x}\right) dt.\label{Fourierexpansionsinecoeff}%
\end{align}
The proof of these facts is analogous to that in the case of Eq.
(\ref{schrodingerregular}), see \cite{neumann} or \cite[Ch. 9]%
{directandinverse}.

\begin{remark}
\label{remarkfirstcoeffcosine}  Since $g_{0}(x)=2a_{0}(x)$, then
$g_{0}(x)=e_{\mathfrak{I}_{N}}^{h}(0,x)-1$. Since $e_{\mathfrak{I}_{N}}%
^{h}(0,x)$ is the solution of (\ref{Schrwithdelta}) with $\rho=0$ satisfying
$e_{\mathfrak{I}_{N}}^{h}(0,0)=1$, $(e_{\mathfrak{I}_{N}}^{h})^{\prime
}(0,0)=h$, hence by Remark \ref{remarkunicityofsol}, $e_{\mathfrak{I}_{N}}%
^{h}(0,x)=c_{\mathfrak{I}_{N}}^{h}(0,x)$ and
\begin{equation}
\label{relationfirstcoeff}g_{0}(x)= c_{\mathfrak{I}_{N}}^{h}(0,x)-1.
\end{equation}
On the other hand, for the coefficient $s_{0}(x)$ we have the relation
\[
s_{0}(x)= 3\int_{0}^{x}H_{\mathfrak{I}_{N}}(x,t)P_{1}\left( \frac{t}{x}\right)
dt= \frac{3}{x}\int_{0}^{x}H_{\mathfrak{I}_{N}}(x,t)tdt.
\]
Since $\frac{\sin(\rho x)}{\rho}\big{|}_{x=0}=x$, from (\ref{transmsinegral})
we have
\begin{equation}
\label{relationcoeffs0}s_{0}(x)= 3\left( \frac{s_{\mathfrak{I}_{N}}(0,x)}%
{x}-1\right) .
\end{equation}

\end{remark}

For every $n\in\mathbb{N}_{0}$ we write the Legendre polynomial $P_{n}(z)$ in
the form $P_{n}(z)=\sum_{k=0}^{n}l_{k,n}z^{k}$. Note that if $n$ is even,
$l_{k,n}=0$ for odd $k$, and $P_{2n}(z)=\sum_{k=0}^{n}\tilde{l}_{k,n}z^{2k}$
with $\tilde{l}_{k,n}=l_{2k,2n}$. Similarly $P_{2n+1}(z)=\sum_{k=0}^{n}\hat
{l}_{k,n}z^{2k+1}$ with $\hat{l}_{k,n}=l_{2k+1,2n+1}$. With this notation we
write an explicit formula for the coefficients (\ref{FourierLegendreCoeff1})
of the canonical transmutation kernel $K_{\mathfrak{J}_{N}}^{f}(x,t)$.

\begin{proposition}
\label{propcoeffwithformalpowers}  The coefficients $\{a_{n}(x)\}_{n=0}%
^{\infty}$ of the Fourier-Legendre expansion (\ref{FourierLegendreSeries1}) of
the canonical transmutation kernel $K_{\mathfrak{I}_{N}}^{f}(x,t)$ are given
by
\begin{equation}
\label{CoeffFourier1formalpowers}a_{n}(x)= \left( n+\frac{1}{2}\right) \left(
\sum_{k=0}^{n}l_{k,n}\frac{\varphi_{f}^{(k)}(x)}{x^{k}}-1\right) .
\end{equation}
The coefficients of the canonical cosine and sine kernels satisfy the
following relations for all $n\in\mathbb{N}_{0}$
\begin{align}
g_{n}(x)  & = (4n+1)\left( \sum_{k=0}^{n}\tilde{l}_{k,n}\frac{\varphi
_{f}^{(2k)}(x)}{x^{2k}}-1\right) ,\label{CoeffFourierCosineformalpowers}\\
s_{n}(x)  & = (4n+3)\left( \sum_{k=0}^{n}\hat{l}_{k,n}\frac{\varphi
_{f}^{(2k+1)}(x)}{x^{2k+1}}-1\right) ,\label{CoeffFourierSineformalpowers}%
\end{align}

\end{proposition}

\begin{proof}
From (\ref{FourierLegendreCoeff1}) we have
\begin{align*}
a_{n}(x)  &  = \left( n+\frac{1}{2}\right) \int_{-x}^{x}K_{\mathfrak{I}_{N}%
}^{f}(x,t)\left( \sum_{k=0}^{n}l_{k,n}\left( \frac{t}{x}\right) ^{k}\right)
dt\\
&  = \left( n+\frac{1}{2}\right) \sum_{k=0}^{n}\frac{l_{k,n}}{x^{k}}\int
_{0}^{x}K_{\mathfrak{I}_{N}}^{f}(x,t)t^{k}dt\\
& = \left( n+\frac{1}{2}\right) \sum_{k=0}^{n}\frac{l_{k,n}}{x^{k}}\left(
\mathbf{T}_{\mathfrak{I}_{N}}^{f}\left[ x^{k}\right] -x^{k} \right) .
\end{align*}
Hence (\ref{CoeffFourier1formalpowers}) follows from Theorem
\ref{theoretransprop} and that $P_{n}(z)=1$. Since $g_{n}(x)=2a_{2n}(x)$,
$s_{n}(x)=2a_{2n+1}(x)$, $l_{2k+1,2n}=0$, $l_{2k,2n+1}=0$ and $l_{2k,2n}%
=\tilde{l}_{k,n}$,$l_{2k+1,2n+1}=\hat{l}_{k,n}$, we obtain
(\ref{CoeffFourierCosineformalpowers}) and (\ref{CoeffFourierSineformalpowers}).
\end{proof}

\begin{remark}
\label{remarkcoefficientsaregood}  By Remark \ref{remarkformalpowersasymp},
formula (\ref{CoeffFourier1formalpowers}) is well defined at $x=0$. Note that
$x^{n}a_{n}(x)$ belongs to $\mathcal{D}_{2}\left( \mathbf{L}_{q,\mathfrak{I}%
_{N}}\right) $ for all $n\in\mathbb{N}_{0}$.
\end{remark}

\subsection{Representation of the solutions as Neumann series of Bessel
functions}

Similarly to the case of the regular Eq. (\ref{regularSch}) \cite{neumann}, we
obtain a representation for the solutions in terms of Neumann series of Bessel
functions (NSBF). For $M\in\mathbb{N}$ we define
\[
K_{\mathfrak{I}_{N},M}^{h}(x,t):= \sum_{n=0}^{M}\frac{a_{n}(x)}{x}P_{n}\left(
\frac{t}{x}\right) ,
\]
that is, the $M$-partial sum of (\ref{FourierLegendreSeries1}).

\begin{theorem}
\label{ThNSBF1}  The solutions $c_{\mathfrak{I}_{N}}^{h}(\rho,x)$ and
$s_{\mathfrak{I}_{N}}(\rho,x)$ admit the following NSBF representations
\begin{align}
c_{\mathfrak{I}_{N}}^{h}(\rho,x)= \cos(\rho x)+\sum_{n=0}^{\infty}%
(-1)^{n}g_{n}(x)j_{2n}(\rho x),\label{NSBFcosine}\\
s_{\mathfrak{I}_{N}}(\rho,x)= \frac{\sin(\rho x)}{\rho}+\frac{1}{\rho}%
\sum_{n=0}^{\infty}(-1)^{n}s_{n}(x)j_{2n+1}(\rho x),\label{NSBFsine}%
\end{align}
where $j_{\nu}$ stands for the spherical Bessel function $j_{\nu}%
(z)=\sqrt{\frac{\pi}{2z}}J_{\nu+\frac{1}{2}}(z)$ (and $J_{\nu}$ stands for the
Bessel function of order $\nu$). The series converge pointwise with respect to
$x$ in $(0,b]$ and uniformly with respect to $\rho$ on any compact subset of
the complex $\rho$-plane. Moreover, for $M\in\mathbb{N}$ the functions
\begin{align}
c_{\mathfrak{I}_{N},M}^{h}(\rho,x)= \cos(\rho x)+\sum_{n=0}^{M}(-1)^{n}%
g_{n}(x)j_{2n}(\rho x),\label{NSBFcosineaprox}\\
s_{\mathfrak{I}_{N},M}(\rho,x)= \frac{\sin(\rho x)}{\rho}+\frac{1}{\rho}%
\sum_{n=0}^{M}(-1)^{n}s_{n}(x)j_{2n+1}(\rho x),\label{NSBFsineaprox}%
\end{align}
obey the estimates
\begin{align}
|c_{\mathfrak{I}_{N}}^{h}(\rho,x)-c_{\mathfrak{I}_{N},M}^{h}(\rho,x)|  &
\leqslant 2\epsilon_{2M}(x)\sqrt{\frac{\sinh(2bC)}{C}}%
,\label{estimatessinecosineNSBF}\\
|\rho s_{\mathfrak{I}_{N}}(\rho,x)-\rho s_{\mathfrak{I}_{N},M}(\rho,x)|  &
\leqslant 2\epsilon_{2M+1}(x)\sqrt{\frac{\sinh(2bC)}{C}}%
,\label{estimatessinesineNSBF}%
\end{align}
for any $\rho\in\mathbb{C}$ belonging to the strip $|\operatorname{Im}%
\rho|\leqslant C$, $C>0$, and where\newline$\epsilon_{M}(x)=\|K_{\mathfrak{I}%
_{N}}^{h}(x,\cdot)-K_{\mathfrak{I}_{N},2M}^{h}(x,\cdot)\|_{L_{2}(-x,x)}$.
\end{theorem}

\begin{proof}
We show the results for the solution $c_{\mathfrak{I}_{N}}^{h}(\rho,x)$ (the
proof for $s_{\mathfrak{I}_{N}}(\rho,x)$ is similar). Substitution of the
Fourier-Legendre series (\ref{Fourierexpansioncosine}) in
(\ref{transmcosinegral}) leads us to
\begin{align*}
c_{\mathfrak{I}_{N}}^{h}(\rho,x)  &  = \cos(\rho x)+\int_{0}^{x}\left(
\sum_{n=0}^{\infty}\frac{g_{n}(x)}{x}P_{2n}\left( \frac{t}{x}\right) \right)
\cos(\rho t)dt\\
& = \cos(\rho x)+\sum_{n=0}^{\infty}\frac{g_{n}(x)}{x}\int_{0}^{x}P_{2n}\left(
\frac{t}{x}\right) \cos(\rho t)dt
\end{align*}
(the exchange of the integral with the summation is due to the fact that the
integral is nothing but the inner product of the series with the function
$\overline{\cos(\rho t)}$ and the series converges in $L_{2}(0,x)$). Using
formula 2.17.7 in \cite[pp. 433]{prudnikov}
\[
\int_{0}^{a}\left\{
\begin{matrix}
P_{2n+1}\left( \frac{y}{a}\right) \cdot\sin(by)\\
P_{2n}\left( \frac{y}{a}\right) \cdot\cos(by)\\
\end{matrix}
\right\} dy= (-1)^{n}\sqrt{\frac{\pi a}{2b}}J_{2n+\delta+\frac{1}{2}}(ab),
\quad\delta=\left\{
\begin{matrix}
1\\
0
\end{matrix}
\right\} , \; a>0,
\]
we obtain the representation (\ref{NSBFcosine}). Take $C>0$ and $\rho
\in\mathbb{C}$ with $|\operatorname{Im}\rho|\leqslant C$. For $M\in\mathbb{N}$
define $G_{\mathfrak{I}_{N},M}^{h}(x,t):= K_{\mathfrak{I}_{N},2M}%
^{h}(x,t)-K_{\mathfrak{I}_{N},2M}^{h}(x,-t)= \sum_{n=0}^{M}\frac{g_{n}(x)}%
{x}P_{2n}\left( \frac{t}{x}\right) $, the $M$-th partial sum of
(\ref{Fourierexpansioncosine}). Then
\[
c_{\mathfrak{I}_{N},M}^{h}(\rho,x) = \cos(\rho x)+\int_{0}^{x}G_{\mathfrak{I}%
_{N},M}^{h}(x,t)\cos(\rho t)dt.
\]
Using the Cauchy-Bunyakovsky-Schwarz inequality we obtain
\begin{align*}
|c_{\mathfrak{I}_{N}}^{h}(\rho,x)-C_{\mathfrak{I}_{N},M}^{h}(\rho,x)|  &  =
\left| \int_{0}^{x}\left( G_{\mathfrak{I}_{N}}^{h}(x,t)-G_{\mathfrak{I}_{N}%
,M}^{h}(x,t)\right) \cos(\rho t)dt\right| \\
& = \left| \left\langle \overline{G_{\mathfrak{I}_{N}}^{h}%
(x,t)-G_{\mathfrak{I}_{N},M}^{h}(x,t)}, \cos(\rho t) \right\rangle
_{L_{2}(0,x)} \right| \\
&  \leqslant\|G_{\mathfrak{I}_{N}}^{h}(x,\cdot)-G_{\mathfrak{I}_{N},M}%
^{h}(x,\cdot)\|_{L_{2}(0,x)}\|\cos(\rho t)\|_{L_{2}(0,x)}.
\end{align*}
Since $\|K_{\mathfrak{I}_{N}}^{h}(x,\cdot)-K_{\mathfrak{I}_{N},2M}^{h}%
(x,\cdot)\|_{L_{2}(-x,x)}=\frac{1}{2}\|G_{\mathfrak{I}_{N}}^{h}(x,\cdot
)-G_{M,n}^{h}(x,\cdot)\|_{L_{2}(0,x)}$,
\begin{align*}
\int_{0}^{x}|\cos(\rho t)|^{2}dt &  \leqslant\frac{1}{4}\int_{0}^{x}\left(
|e^{i\rho t}|+|e^{-i\rho t}|\right) ^{2}dt \leqslant\frac{1}{2}\int_{0}^{x}
\left( e^{-2t\operatorname{Im\rho}}+ e^{2t\operatorname{Im\rho}}\right) dt\\
&  = \int_{-x}^{x}e^{-2\operatorname{Im}\rho t}dt = \frac{\sinh
(2x\operatorname{Im}\rho)}{\operatorname{Im}\rho}%
\end{align*}
and the function $\frac{\sinh(\xi x)}{\xi}$ is monotonically increasing in
both variables when $\xi,x\geqslant0$, we obtain
(\ref{estimatessinecosineNSBF}).
\end{proof}

Given $H\in\mathbb{C}$, we look for a pair of solutions $\psi_{\mathfrak{I}%
_{N}}^{H}(\rho,x)$ and $\vartheta_{\mathfrak{I}_{N}}(\rho,x)$ of
(\ref{Schrwithdelta}) satisfying the conditions
\begin{align}
\psi_{\mathfrak{I}_{N}}^{H}(\rho,b)=1,  &\quad  (\psi_{\mathfrak{I}_{N}}%
^{H})^{\prime}(\rho,b)=-H,\label{conditionspsi}\\
\vartheta_{\mathfrak{I}_{N}}(\rho,b)=0, &\quad   \vartheta_{\mathfrak{I}_{N}%
}^{\prime}(\rho,b)=1.\label{conditionstheta}%
\end{align}

\begin{theorem}
The solutions $\psi_{\mathfrak{I}_{N}}^{H}(\rho,x)$ and $\vartheta
_{\mathfrak{I}_{N}}(\rho,x)$ admit the integral representations
\begin{align}
\psi_{\mathfrak{I}_{N}}^{H}(\rho,x)  &  =  \cos(\rho(b-x))+\int_{x}%
^{b}\widetilde{G}_{\mathfrak{I}_{N}}^{H}(x,t)\cos(\rho
(b-t))dt,\label{solpsiintrepresentation}\\
\vartheta_{\mathfrak{I}_{N}}(\rho,x)  &  =  \frac{\sin(\rho(b-x))}{\rho}%
+\int_{x}^{b}\widetilde{S}_{\mathfrak{I}_{N}}^{H}(x,t)\frac{\sin(\rho
(b-t))}{\rho}dt,\label{solthetaintrepresentation}%
\end{align}
where the kernels $\widetilde{G}_{\mathfrak{I}_{N}}^{H}(x,t)$ and
$\widetilde{S}_{\mathfrak{I}_{N}}(x,t)$ are defined in $\Omega$ and satisfy
$\widetilde{G}_{\mathfrak{I}_{N}}^{H}(x,\cdot), \widetilde{S}_{\mathfrak{I}%
_{N}}(x,\cdot) \in L_{2}(0,x)$ for all $x\in(0,b]$. In consequence, the
solutions $\psi_{\mathfrak{I}_{N}}^{H}(\rho,x)$ and $\vartheta_{\mathfrak{I}%
_{N}}(\rho,x)$ can be written as NSBF
\begin{equation}
\label{NSBFforpsi}\psi_{\mathfrak{I}_{N}}^{H}(\rho,x)= \cos(\rho
(b-x))+\sum_{n=0}^{\infty}(-1)^{n}\tau_{n}(x)j_{2n}(\rho(b-x)),
\end{equation}
\begin{equation}
\label{NSBFfortheta}\vartheta_{\mathfrak{I}_{N}}(\rho,x)= \frac{\sin
(\rho(b-x))}{\rho}+\sum_{n=0}^{\infty}(-1)^{n}\zeta_{n}(x)j_{2n}(\rho(b-x)),
\end{equation}
with some coefficients $\{\tau_{n}(x)\}_{n=0}^{\infty}$ and $\{\zeta
_{n}(x)\}_{n=0}^{\infty}$.
\end{theorem}

\begin{proof}
We prove the results for $\psi_{\mathfrak{I}_{N}}^{H}(\rho,x)$ (the proof for
$\vartheta_{\mathfrak{I}_{N}}(\rho,x)$ is similar). Set $y(\rho,x)=
\psi_{\mathfrak{I}_{N}}^{H}(\rho,b-x)$. Note that $y(\rho, 0)=1$, $y^{\prime
}(\rho,0)= H$ and for $\phi\in C_{0}^{\infty}(0,b)$ we have%

\begin{align*}
(y^{\prime\prime2}y(x), \phi(x))_{C_{0}^{\infty}(0,b)}  &  = (\psi
_{\mathfrak{I}_{N}}^{H}(\rho,x),\phi^{\prime\prime2}\phi(b-x) )_{C_{0}%
^{\infty}(0,b)}\\
& = (q(x)\psi_{\mathfrak{I}_{N}}^{H}(\rho,x), \phi(b-x))_{C_{0}^{\infty}%
(0,b)}+\sum_{k=0}^{N}\alpha_{k}\psi_{\mathfrak{I}_{N}}^{H}(\rho,x_{k}%
)\phi(b-x_{k})\\
& = (q(b-x)y(x),\phi(x))_{C_{0}^{\infty}(0,b)}+\sum_{k=0}^{N}\alpha
_{k}y(b-x_{k})\phi(b-x_{k}),
\end{align*}
that is, $\psi_{\mathfrak{I}_{N}}^{H}(\rho,x)$ is a solution of
(\ref{Schrwithdelta}) iff $y(x)=\psi_{\mathfrak{I}_{N}}^{H}(\rho,b-x)$ is a
solution of
\begin{equation}
\label{reversedequation}-y^{\prime\prime}(x)+\left( q(b-x)+\sum_{k=0}%
^{N}\alpha_{k}\delta(x-(b-x_{k}))\right) y(x)=\rho^{2}y(x).
\end{equation}
Since $0<b-x_{N}<\cdots<b-x_{0}<b$, hence (\ref{reversedequation}) is of the
type (\ref{Schrwithdelta}) with the point interactions $\mathfrak{I}_{N}%
^{*}=\{(b-x_{N-j},\alpha_{N-j})\}_{j=0}^{N}$ and $\psi_{\mathfrak{I}_{N}}%
^{H}(\rho,b-x)$ is the corresponding solution $c_{\mathfrak{I}_{N}^{*}}%
^{H}(\rho,x)$ for (\ref{reversedequation}). Hence
\begin{equation}
\label{integralrepresentationpsi}\psi_{\mathfrak{I}_{N}}^{H}(\rho,b-x)=
\cos(\rho x)+ \int_{0}^{x}G_{\mathfrak{I}_{N}^{*}}^{H}(x,t)\cos(\rho t)dt
\end{equation}
for some kernel $G_{\mathfrak{I}_{N}^{*}}^{H}(x,t)$ defined on $\Omega$ with
$\widetilde{G}_{\mathfrak{I}_{N}}^{H}(x,\cdot)\in L_{2}(0,x)$ for $x\in(0,b]$.
Thus,
\begin{align*}
\psi_{\mathfrak{I}_{N}}(\rho,x)  &  =\cos(\rho(b-x))+\int_{0}^{b-x}%
G_{\mathfrak{I}_{N}^{*}}^{H}(b-x,t)\cos(\rho t)dt\\
& =\psi_{\mathfrak{I}_{N}}(\rho,x)=\cos(\rho(b-x))+\int_{x}^{b}G_{\mathfrak{I}%
_{N}^{*}}^{H}(b-x,b-t)\cos(\rho(b-t))dt,
\end{align*}
where the change of variables $x\mapsto b-x$ was used. Hence we obtain
(\ref{solpsiintrepresentation}) with $\widetilde{G}_{\mathfrak{I}_{N}^{*}}%
^{H}(x,t)=G_{\mathfrak{I}_{N}^{*}}^{H}(b-x,b-t)$ In consequence, by Theorem
\ref{ThNSBF1} we obtain (\ref{NSBFforpsi}).
\end{proof}

\begin{remark}
As in Remark \ref{remarkfirstcoeffcosine}
\begin{equation}
\label{firstcoeffpsi}\tau_{0}(x)= \psi_{\mathfrak{I}_{N}}^{H}(0,x)-1
\quad\mbox{and }\; \zeta_{0}(x)=3\left( \frac{\vartheta_{\mathfrak{I}_{N}%
}(0,x)}{b-x}-1\right) .
\end{equation}

\end{remark}

\begin{remark}
\label{remarkentirefunction}  Let $\lambda\in\mathbb{C}$ and $\lambda=\rho
^{2}$. 

\begin{itemize}

\item[(i)] The functions $\widehat{s}_{k}(\rho,x-x_{k})$ are entire with
respect to $\rho$. Then from (\ref{generalsolcauchy}) $c_{\mathfrak{I}_{N}%
}^{h}(\rho,x)$, $s_{\mathfrak{I}_{N}}(\rho,x)$ and $\psi_{\mathfrak{I}_{N}%
}^{H}(\rho, x)$ are entire as well. 

\item[(ii)] Suppose that $q$ is real valued and $\alpha_{0}, \dots, \alpha
_{N}, u_{0}, u_{1}\in\mathbb{R}$. If $u(\lambda,x)$ is a solution of
$u^{(k)}(\lambda,0)=u_{k}$, $k=0,1$, then by the uniqueness of the Cauchy
problem $\overline{u(\lambda,x)}=u(\overline{\lambda},x)$. In particular, for
$\rho, h, H\in\mathbb{R}$, the solutions $c_{\mathfrak{I}_{N}}^{h}(\rho,x)$,
$s_{\mathfrak{I}_{N}}(\rho,x)$ and $\psi_{\mathfrak{I}_{N}}^{H}(\rho, x)$ are
real valued. 
\end{itemize}
\end{remark}

\subsection{A recursive integration procedure for the coefficients $\{a_n(x)\}_{n=0}^{\infty}$}

Similarly to the case of the regular Schr\"odinger equation \cite{directandinverse,neumann,neumann2}, we derive formally a recursive integration procedure for computing the Fourier-Legendre coefficients $\{a_n(x)\}_{n=0}^{\infty}$ of the canonical transmutation kernel $K_{\mathfrak{J}_N}^f(x,t)$. Consider the sequence of functions $\sigma_n(x):=x^na_n(x)$ for $n\in \mathbb{N}_0$. According to Remark \ref{remarkcoefficientsaregood}, $\{\sigma_n(x)\}_{n=0}^{\infty}\subset \mathcal{D}_2\left(\mathbf{L}_{q,\mathfrak{J}_N}\right)$.

\begin{remark}\label{Remarksigmafunctionproperties}
	\begin{itemize}
		\item[(i)]  By Remark \ref{remarkfirstcoeffcosine}, 
		\begin{equation}\label{sigma0}
			\sigma_0(x)=\frac{f(x)-1}{2}.
		\end{equation}
		\item[(ii)] By (\ref{CoeffFourier1formalpowers}), $a_1(x)=\frac{3}{2}\left(\frac{\varphi_f^{(1)}(x)}{x}-1\right)$. Thus, from (\ref{formalpowers}) and (\ref{auxiliaryformalpowers}) we have
		\begin{equation}\label{sigma1}
			\sigma_1(x)=\frac{3}{2}\left(f(x)\int_0^x\frac{dt}{f^2(t)}-x\right).
		\end{equation}
	\item[(iii)] For $n\geqslant 2$, $\sigma_n(0)=0$, and by (\ref{CoeffFourier1formalpowers}) we obtain 
	\begin{align*}
		D\sigma_n(x) & = \left(n+\frac{1}{2}\right)\sum_{k=0}^{n}l_{k,n}D\left(x^{n-k}\varphi_f^{(k)}(x)\right)\\
		& =\left(n+\frac{1}{2}\right)\left(\sum_{k=0}^{n-1}l_{k,n} (n-k)x^{n-k-1}\varphi_f^{(k)}(x)+\sum_{k=0}^{n}l_{k,n}x^{n-k}D\varphi_f^{(k)}(x)\right). 
	\end{align*}
	By (\ref{formalpowersderivative}) and (\ref{auxiliaryformalpowers}), $D\varphi_f^{(k)}(0)=0$ for $k\geqslant 1$. Hence, $\sigma_n'(0)=0$.
	\end{itemize}
\end{remark}
Denote by $c_{\mathfrak{J}_N}^f(\rho,x)$ the solution of (\ref{Schrwithdelta}) satisfying (\ref{cosinegralinitialconds}) with $h=f'(0)$. On each interval $[x_k,x_{k+1}]$, $k=0, \cdots, N$, $c_{\mathfrak{J}_N}^f(\rho,x)$ is a solution of the regular equation (\ref{schrodingerregular}). In \cite[Sec. 6]{neumann}  by substituting the Neumann series (\ref{NSBFcosine}) of $c_{\mathfrak{J}_N}^f(\rho,x)$ into Eq. (\ref{schrodingerregular}) it was proved that the functions $\{\sigma_{2n}(x)\}_{n=0}^{\infty}$ must satisfy, at least formally, the recursive relations
\begin{equation}\label{firstrecursiverelation}
	\mathbf{L}_q\sigma_{2n}(x) = \frac{4n+1}{4n-3}x^{4n-1}\mathbf{L}_q\left[\frac{\sigma_{2n-2}(x)}{x^{4n-3}}\right], \quad x_k<x<x_k
\end{equation}
for $k=0,\cdots, N$. Similarly, substitution of the Neumann series (\ref{NSBFsine}) of $s_{\mathfrak{J}_N}(\rho,x)$ into (\ref{schrodingerregular}) leads to the equalities
\begin{equation}\label{secondrecursiverelation}
	\mathbf{L}_q\sigma_{2n+1}(x) = \frac{4n+3}{4n+1}x^{4n+3}\mathbf{L}_q\left[\frac{\sigma_{2n+1}(x)}{x^{4n+1}}\right], \quad x_k<x<x_k.
\end{equation}
    Taking into account that $\sigma_n\in \mathcal{D}_2\left(\mathbf{L}_{q,\mathfrak{J}_N}\right)$ and combining (\ref{firstrecursiverelation}), by Remark \ref{Remarksigmafunctionproperties}(iii) and (\ref{secondrecursiverelation}) we obtain that the functions $\sigma_n(x)$, $n\geqslant 2$, must satisfy (at least formally) the following Cauchy problems
\begin{equation}\label{recurrencerelationsigma}
\begin{cases} 
	\displaystyle \mathbf{L}_{q,\mathfrak{J}_N}\sigma_n(x) = \frac{2n+1}{2n-3}x^{2n-1}\mathbf{L}_q\left[\frac{\sigma_{n-2}(x)}{x^{2n-3}}\right], \quad 0<x<b, \\
	\sigma_n(0) =\sigma'_n(0)=0.
\end{cases} 
\end{equation}

\begin{remark}\label{remarkcontinuousquotient}
	If $g\in \mathcal{D}_2\left(L_{q,\mathfrak{J}_N}\right)$, then $\frac{g}{f}\in H^2(0,b)$.
	
	Indeed, $\frac{g}{f}\in C[0,b]$, and the jump of the derivative at $x_k$ is given by
	\begin{align*}
		\left(\frac{g}{f}\right)'(x_k+)-	\left(\frac{g}{f}\right)'(x_k-) & = \frac{g'(x_k+)f(x_k)-f'(x_k+)g(x_k)}{f^2(x_k)}-\frac{g'(x_k-)f(x_k)-f'(x_k-)g(x_k)}{f^2(x_k)}\\
		& = \frac{1}{f^2(x_k)}\left[\left(g'(x_k+)-g'(x_k-)\right)f(x_k)-g(x_k)\left(f'(x_k+)-f'(x_k-)\right) \right]\\
		& = \frac{1}{f^2(x_k)}\left[ \alpha_kg(x_k)f(x_k)-\alpha_kg(x_k)f(x_k)  \right]=0.
	\end{align*}
Hence $\frac{g}{f}\in AC[0,b]$, and then $\frac{g}{f}\in H^2(0,b)$.
\end{remark}

\begin{proposition}
	The sequence $\{\sigma_n(x)\}_{n=0}^{\infty}$ satisfying the recurrence relation (\ref{recurrencerelationsigma}) for $n\geqslant 2$, with $\sigma_0(x)=\frac{f(x)-1}{2}$ and $\sigma_1(x)= \frac{3}{2}\left(f(x)\int_0^x\frac{dt}{f^2(t)}-x\right)$,  is given by
\begin{equation}\label{recurrenceintegralsigma}
  \sigma_n(x)= \frac{2n+1}{2n-3}\left( x^2\sigma_{n-2}(x)+2(2n-1)\theta_n(x)\right), \quad n\geqslant 2,	
\end{equation}
where
\begin{equation}
	\theta_n(x):= \int_0^x\left(\eta_n(t)-tf(t)\sigma_{n-2}(t)\right)\frac{dt}{f^2(t)}, \quad n\geqslant 2,
\end{equation}
and 
\begin{equation}
	\eta_n(x):= \int_0^x\left((n-1)f(t)+tf'(t)\right)\sigma_{n-2}(t)dt, \quad n\geqslant 2.
\end{equation}
\end{proposition}
\begin{proof}
	Set $g\in \mathcal{D}_2\left(\mathbf{L}_{q,\mathfrak{J}_N}\right)$ and $n\geqslant 2$. Consider the Cauchy problem 
	\begin{equation}\label{auxcauchyproblem2}
		\begin{cases} 
			\displaystyle \mathbf{L}_{q,\mathfrak{J}_N}u_n(x) = \frac{2n+1}{2n-3}x^{2n-1}\mathbf{L}_q\left[\frac{g(x)}{x^{2n-3}}\right], \quad 0<x<b, \\
			u_n(0) =u'_n(0)=0.
		\end{cases} 
	\end{equation}
	By formula (\ref{solutionnonhomogeneouseq}) and the Polya factorization $\mathbf{L}_q= -\frac{1}{f}Df^2D\frac{1}{f}$ we obtain that the unique solution of the Cauchy problem (\ref{auxcauchyproblem2}) is given by
	\[
	u_n(x)= \frac{2n+1}{2n-3} f(x)\int_0^x \frac{1}{f^2(t)}\left(\int_0^t s^{2n-1} Df^2(s)D\left[ \frac{g(s)}{s^{2n-3}f(s)}\right]ds \right)dt.
	\]
	
 Consider an antiderivative $\int s^{2n-1} Df^2(s)D\left[ \frac{g(s)}{s^{2n-3}f(s)}\right]ds$. Integration by parts gives
 \begin{align*}
 	\int s^{2n-1} Df^2(s)D\left[ \frac{g(s)}{s^{2n-3}f(s)}\right]ds & = s^{2n-1}f^2(s)D\left(\frac{g(s)}{s^{2n-3}f(s)}\right)-(2n-1)sf(s)g(s) \\
 	& \quad + \int \left((2n-1)(2n-2)f(s)+2(2n-1)sf'(s)\right)g(s)ds.
 \end{align*}
Note that
\begin{align*}
	s^{2n-1}f^2(s)D\left(\frac{g(s)}{s^{2n-3}f(s)}\right) & = s^{2n-1}f^2(s)\frac{D\left(\frac{g(s)}{f(s)}\right)}{s^{2n-3}}-s^{2n-1}f^2(s)\frac{\frac{g(s)}{f(s)}}{s^{4n-6}}(2n-3)s^{2n-4}\\
	& = s^2f^2(s)D\left(\frac{g(s)}{f(s)}\right)-(2n-3)sf(s)g(s).
\end{align*}

Since $g\in \mathcal{D}_2\left(\mathbf{L}_{q,\mathfrak{J}_N}\right)$, by Remark \ref{remarkcontinuousquotient}, $D\left(\frac{g(s)}{f(s)}\right)$ is continuous in $[0,b]$. Thus,
\begin{align*}
	\int s^{2n-1} Df^2(s)D\left[ \frac{g(s)}{s^{2n-3}f(s)}\right]ds 
	& = s^2f^2(s)D\left(\frac{g(s)}{f(s)}\right)-(4n-4)sf(s)g(s)\\
	&\quad  + 2(2n-1)\int\left((n-1)f(s)+sf'(s)\right)ds
\end{align*}
is well defined at $s=0$ and is continuous in $[0,b]$. Then we obtain that 
\begin{align*}
\Phi(t) & :=  	\int_0^t s^{2n-1} Df^2(s)D\left[ \frac{g(s)}{s^{2n-3}f(s)}\right]ds \\
 & = t^2f^2(t)D\left(\frac{g(t)}{f(t)}\right)-(4n-4)tf(t)g(t)  + 2(2n-1)\Theta_n[g](t),
\end{align*}
with $H_n[g](t):= \displaystyle \int_0^t \left((n-1)f(s)+sf'(s)\right)g(s)ds$, 
is a continuous function in $[0,b]$. Now, 
\begin{align*}
	\int_0^x \Phi(t)\frac{dt}{f^2(t)}  &= \int_0^x t^2D\left[\frac{g(t)}{f(t)}\right]dt-(4n-4)\int_{0}^{t} t\frac{g(t)}{f(t)}dt+2(2n-1)H_n[g](t)\\
	& = x^2\frac{g(x)}{f(x)}-2(2n-1)\int_0^x\left[H_n[g](t)-tf(t)g(t)\right]dt.
\end{align*}
Hence
\begin{equation}\label{auxsol2}
	u_n(x)= \frac{2n+1}{2n-3}\left(x^2g(x)-2(2n-1)\Theta_n[g](x)\right),
\end{equation}
 with $\Theta_n[g](x):= \displaystyle \int_0^x\left[H_n[g](t)-tf(t)g(t)\right]dt$.

 Finally, since $\sigma_0, \sigma_1\in \mathcal{D}_2\left(\mathbf{L}_{q,\mathfrak{J}_{N}}\right)$, formula (\ref{recurrenceintegralsigma}) is obtained for all $n\geqslant 2$ by induction, taking $g=\sigma_{2n-2}$ in (\ref{auxcauchyproblem2}) and $\eta_n(x)=H_n[\sigma_{n-2}](x)$, $\theta_n(x)=\Theta_n[\sigma_{n-2}](x)$ in (\ref{auxsol2}). 
\end{proof}

Integral relations of type (\ref{recurrenceintegralsigma}) are effective for the numerical computation of the partial sums (\ref{NSBFcosineaprox}) and (\ref{NSBFsineaprox}), as seen in \cite{neumann,neumann2}.
\section{Integral representation for the derivative}

Since $e_{\mathfrak{I}_{N}}^{h}(\rho, \cdot)\in AC[0, b]$, it is worthwhile
looking for an integral representation of the derivative of $e_{\mathfrak{I}%
_{N}}^{h}(\rho,x)$. Differentiating (\ref{SoleGral}) we obtain
\begin{align*}
(e_{\mathfrak{I}_{N}}^{h})^{\prime}(\rho,x)  &  =\widetilde{e}_{h}^{\prime
}(\rho,x)+\sum_{k=0}^{N}\alpha_{k}\widetilde{e}_{h}(\rho, x_{k})H(x-x_{k}%
)\widehat{s}^{\prime}_{k} (\rho, x-x_{k})\\
&  \quad+\sum_{J\in\mathcal{I}_{N}}\alpha_{J} H(x-x_{j_{|J|}})\widetilde
{e}_{h}(\rho,x_{j_{1}})\left( \prod_{l=1}^{|J|-1}\widehat{s}_{j_{l}}%
(\rho,x_{j_{l+1}}-x_{j_{l}})\right) \widehat{s}_{j_{|J|}}^{\prime}%
(\rho,x-x_{j_{|J|}}).
\end{align*}
Differentiating (\ref{representationsinegeneral1}) and using that $\widehat
{H}_{k}(x,x)=\frac{1}{2}\int_{0}^{x}q(t+x_{k})dt$, we obtain
\begin{align*}
\widehat{s}_{k}^{\prime}(\rho,x)  &  = \cos(\rho x)+\frac{1}{2}\frac{\sin(\rho
x)}{\rho}\int_{0}^{x}q(t+x_{k})dt+\int_{0}^{x}\partial_{x}\widehat{H}%
_{k}(x,t)\frac{\sin(\rho t)}{\rho}dt.
\end{align*}
Denote
\begin{equation}
\label{antiderivativeq}w(y,x):= \frac{1}{2}\int_{y}^{x}q(s)ds \quad
\mbox{ for }\; x,y\in[0,b].
\end{equation}
Hence, the derivative $\widehat{s}_{k}^{\prime}(\rho,x-x_{k})$ can be written
as
\begin{equation}
\widehat{s}_{k}^{\prime}(\rho,x-x_{k})= \cos(\rho(x-x_{k}))+\int_{-(x-x_{k}%
)}^{x-x_{k}}\widetilde{K}_{k}^{1}(x,t)e^{i\rho t}dt,
\end{equation}
where $\displaystyle \widetilde{K}_{k}^{1}(x,t)= w(x_{k},x)+\frac{1}{2}%
\int_{|t|}^{x-x_{k}}\partial_{x}\widehat{H}_{k}(x,t)dt$.

On the other hand, differentiation of (\ref{transm1}) and the Goursat
conditions for $\widetilde{K}^{h}(x,t)$ lead to the equality
\begin{equation}
\tilde{e}^{\prime}_{h}(\rho,x)= (i\rho+w(0,x))e^{i\rho x}+h\cos(\rho
x)+\int_{-x}^{x} \partial_{x}\widetilde{K}^{h}(x,t)e^{i\rho t}dt.
\end{equation}
Using the fact that
\[
\cos(\rho A)\int_{-B}^{B}f(t)e^{i\rho t}dt= \int_{-(B+A)}^{B+A}\frac{1}%
{2}\left( \chi_{[-(B+A),B-A]}(t)f(t-A)+\chi_{[A-B,B+A]}(t)f(t+A)\right)
e^{i\rho t}dt
\]
for $A,B>0$ and $f\in L_{2}(\mathbb{R})$ with $\operatorname{Supp}%
(f)\subset[-B,B]$, we obtain
\begin{align*}
\tilde{e}_{h}(\rho, x_{j})\widehat{s}_{k}^{\prime}(\rho,x-x_{k})  &  =e^{i\rho
x_{j}}\cos(\rho(x-x_{k}))+\mathcal{F}\left[  \widehat{K}_{x_{j},x_{k}%
}(x,t)\right] ,
\end{align*}
where
\begin{align*}
\widehat{K}_{x_{j},x_{k}}(x,t)  &  = \chi_{[x_{k}-x-x_{j},x-x_{k}-x_{j}%
]}(t)\widetilde{K}_{k}^{1}(x,t-x_{j})+\chi_{x_{j}}(t)\widetilde{K}^{h}%
(x_{j},t)\ast\chi_{x-x_{k}}(t)\widehat{K}_{k}^{1}(x,t)\\
& \qquad+\frac{1}{2}\chi_{[x_{k}-x_{j}-x,x_{j}-x+x_{k}]}(t)\widehat{K}%
^{h}(x_{j},t-x+x_{k})\\
& \qquad+\frac{1}{2}\chi_{[x-x_{k}-x_{j},x-x_{k}+x_{j}]}(t)\widehat{K}%
^{h}(x_{j},t+x-x_{k})\Big).
\end{align*}
By Lemma \ref{lemaconv} the support of $\widehat{K}_{x_{j},x_{k}}(x,t)$
belongs to $[x_{k}-x-x_{j},x-x_{k}+x_{j}]$. Using the equality
\[
\prod_{l=1}^{|J|-1}\widehat{s}_{j_{l}}(\rho,x_{j_{l+1}}-x_{j_{l}}%
)=\mathcal{F}\left\{ \left( \prod_{l=1}^{|J|-1}\right) ^{\ast}\left(
\chi_{x_{j_{l+1}}-x_{j_{l}}}(t)\widetilde{K}_{k}(x_{j_{l+1}},t)\right)
\right\}
\]
we have
\[
(e_{\mathfrak{I}_{N}}^{h})^{\prime i\rho x}+h\cos(\rho x)+\sum_{k=0}^{N}%
\alpha_{k}H(x-x_{k})e^{i\rho x_{k}}\cos(\rho(x-x_{k}))+\mathcal{F}\left\{
E_{\mathfrak{I}_{N}}^{h}(x,t)\right\}
\]
where
\begin{align*}
E_{\mathfrak{I}_{N}}^{h}(x,t)  &  = \chi_{x}(t)\partial_{x}\widetilde{K}%
^{h}(x,t)+\sum_{k=0}^{N}\alpha_{k} H(x-x_{k})\widehat{K}_{x_{k},x_{k}}(x,t)\\
&  \quad+\sum_{J\in\mathcal{I}_{N}}\alpha_{J} H(x-x_{j_{|J|}}) \widehat
{K}_{x_{j_{1}},x_{j_{|J|}}}(x,t)\ast\left( \prod_{l=1}^{|J|-1}\right) ^{\ast
}\left( \chi_{x_{j_{l+1}}-x_{j_{l}}}(t)\widetilde{K}_{k}(x_{j_{l+1}},t)\right)
.
\end{align*}
Again, by Lemma \ref{lemaconv} the support of $E_{\mathfrak{I}_{N}}^{h}(x,t)$
belongs to $[-x,x]$. Since $e^{i\rho x_{k}}\cos(\rho(x-x_{k}))=\frac{1}%
{2}e^{i\rho x}\left( 1+e^{-2i\rho(x-x_{k})}\right) $, we obtain the following representation.

\begin{theorem}
The derivative $(e_{\mathfrak{I}_{N}}^{h})^{\prime}(\rho,x)$ admits the
integral representation
\begin{align}
(e_{\mathfrak{I}_{N}}^{h})^{\prime}(\rho,x)  &  = \left( i\rho+w(0,x)+\frac
{1}{2}\sigma_{\mathfrak{I}_{N}}(x)\right) e^{i\rho x}+h\cos(\rho x)\nonumber\\
& \quad+\sum_{k=0}^{N}\frac{\alpha_{k}}{2}H(x-x_{k})e^{-2i\rho(x-x_{k})}%
+\int_{-x}^{x}E_{\mathfrak{I}_{N}}^{h}(x,t)e^{i\rho t}%
dt,\label{derivativetransme}%
\end{align}
where $E_{\mathfrak{I}_{N}}^{h}(x,\cdot)\in L_{2}(-x,x)$ for all $x\in(0,b]$.
\end{theorem}

\begin{corollary}
The derivatives of the solutions $c_{\mathfrak{I}_{N}}^{h}(\rho,x)$ and
$s_{\mathfrak{I}_{N}}(\rho,x)$ admit the integral representations
\begin{align}
(c_{\mathfrak{I}_{N}}^{h})^{\prime}(\rho,x)  &  =-\rho\sin(\rho x)+ \left(
h+w(0,x)+\frac{1}{2}\sigma_{\mathfrak{I}_{N}}(x)\right) \cos(\rho
x)\nonumber\\
& \quad+\sum_{k=0}^{N}\frac{\alpha_{k}}{2}H(x-x_{k})\cos(2\rho(x-x_{k}%
))+\int_{0}^{x}M_{\mathfrak{I}_{N}}^{h}(x,t)\cos(\rho
t)dt,\label{derivativetransmcosine}\\
s^{\prime}_{\mathfrak{I}_{N}}(\rho,x)  &  =\cos(\rho x)+ \left( w(0,x)+\frac
{1}{2}\sigma_{\mathfrak{I}_{N}}(x)\right) \frac{\sin(\rho x)}{\rho}\nonumber\\
& \quad-\sum_{k=0}^{N}\alpha_{k}H(x-x_{k})\frac{\sin(2\rho(x-x_{k}))}{2\rho
}+\int_{0}^{x}R_{\mathfrak{I}_{N}}(x,t)\frac{\sin(\rho t)}{\rho}%
dt,\label{derivativetransmsine}%
\end{align}
where
\begin{align}
N_{\mathfrak{I}_{N}}^{h}(x,t)  &  = E_{\mathfrak{I}_{N}}^{h}%
(x,t)+E_{\mathfrak{I}_{N}}^{h}(x,-t)\quad\label{kernelderivativecosine}\\
R_{\mathfrak{I}_{N}}^{h}(x,t)  &  = E_{\mathfrak{I}_{N}}^{h}%
(x,t)-E_{\mathfrak{I}_{N}}^{h}(x,-t),\label{kernelderivativesine}%
\end{align}
defined for $x\in[0,b]$ and $|t|\leqslant x$.
\end{corollary}

\begin{corollary}
\label{CorNSBFforderivatives}  The derivatives of the solutions
$c_{\mathfrak{I}_{N}}^{h}(\rho,x)$ and $s_{\mathfrak{I}_{N}}(\rho,x)$ admit
the NSBF representations
\begin{align}
(c_{\mathfrak{I}_{N}}^{h})^{\prime}(\rho,x)  &  =-\rho\sin(\rho x)+ \left(
h+w(0,x)+\frac{1}{2}\sigma_{\mathfrak{I}_{N}}(x)\right) \cos(\rho
x)\nonumber\\
& \quad+\sum_{k=0}^{N}\frac{\alpha_{k}}{2}H(x-x_{k})\cos(2\rho(x-x_{k}%
))+\sum_{n=0}^{\infty}(-1)^{n}l_{n}(x)j_{2n}(\rho
x),\label{NSBFderivativetransmcosine}\\
s^{\prime}_{\mathfrak{I}_{N}}(\rho,x)  &  =\cos(\rho x)+ \left( w(0,x)+\frac
{1}{2}\sigma_{\mathfrak{I}_{N}}(x)\right) \frac{\sin(\rho x)}{\rho}\nonumber\\
& \quad-\sum_{k=0}^{N}\alpha_{k}H(x-x_{k})\frac{\sin(2\rho(x-x_{k}))}{2\rho
}+\sum_{n=0}^{\infty}(-1)^{n}r_{n}(x)j_{2n+1}(\rho
x),\label{NSBFderivativetransmsine}%
\end{align}
where $\{l_{n}(x)\}_{n=0}^{\infty}$ and $\{r_{n}(x)\}_{n=0}^{\infty}$ are the
coefficients of the Fourier-Legendre expansion of $M_{\mathfrak{I}_{N}}%
^{h}(x,t)$ and $R_{\mathfrak{J}_{N}}(x,t)$ in terms of the even and odd
Legendre polynomials, respectively.
\end{corollary}

\section{Conclusions}

The construction of a transmutation operator that transmute the solutions of
equation $v^{\prime\prime }+\lambda v=0$ into solutions of (\ref{Schrwithdelta}) is
presented. The transmutation operator is obtained from the closed form of the
general solution of equation (\ref{Schrwithdelta}). It was shown how to
construct the image of the transmutation operator on the set of polynomials,
this with the aid of the SPPS method. A Fourier-Legendre series representation
for the integral transmutation kernel is obtained, together with a
representation for the solutions $c_{\mathfrak{I}_{N}}^{h}(\rho,x)$,
$s_{\mathfrak{I}_{N}}(\rho,x)$ and their derivatives as Neumann series of
Bessel functions, together with integral recursive relations for the construction of the Fourier-Legendre coefficients. The series (\ref{NSBFcosine}), (\ref{NSBFsine}),
(\ref{NSBFderivativetransmcosine}), (\ref{NSBFderivativetransmsine}) are
useful for solving direct and inverse spectral problems for
(\ref{Schrwithdelta}), as shown for the regular case
\cite{inverse1,directandinverse,neumann,neumann2}.

\section*{Acknowledgments}

Research was supported by CONACYT, M\'exico via the project 284470.
%%%%%%%%%%%%%%555
\section*{Declarations}
\subsection*{Conflict of interest}
The authors declare no potential conflict of interest.
\subsection*{Data availability}
Data sharing not applicable to this article as no datasets were generated or analysed during the current study.

\end{document}